\documentclass{scrartcl}
\usepackage[utf8]{inputenc}
\usepackage[T1]{fontenc}
\usepackage[UKenglish]{babel}
\usepackage{amsmath}
\usepackage{amsthm}
\usepackage{amssymb}
\usepackage{mathtools}
\usepackage{enumerate}

\usepackage{hyperref}

\DeclarePairedDelimiter\ceil{\lceil}{\rceil}

\DeclareMathOperator{\Lip}{Lip}

\DeclareMathOperator{\diam}{diam}

\newtheorem{proposition}{Proposition}[section]
\newtheorem{corollary}[proposition]{Corollary}
\newtheorem{theorem}[proposition]{Theorem}
\newtheorem{lemma}[proposition]{Lemma}
\theoremstyle{remark}
\newtheorem{remark}[proposition]{Remark}
\newtheorem{example}[proposition]{Example}

\makeatletter
\def\namedlabel#1#2{\begingroup
   \def\@currentlabel{#2}%
   \label{#1}\endgroup
}
\makeatother

\numberwithin{equation}{section}

\title{Generic properties of nonexpansive mappings on unbounded domains}
\author{Christian Bargetz\footnote{Universität Innsbruck, Department of Mathematics, Technikerstraße 13, 6020 Innsbruck, Austria, \texttt{christian.bargetz@uibk.ac.at}} \and Simeon Reich\footnote{Department of Mathematics, The Technion---Israel Institute of Technology,
32000 Haifa, Israel, \texttt{sreich@technion.ac.il}} \and Daylen Thimm\footnote{Universität Innsbruck, Department of Mathematics, Technikerstraße 13, 6020 Innsbruck, Austria, \texttt{daylen.thimm@uibk.ac.at}}}

\begin{document}
\maketitle

\begin{abstract}
  \noindent\textbf{Abstract.} We investigate typical properties of nonexpansive mappings on unbounded complete hyperbolic metric spaces. For two families of metrics of uniform convergence on bounded sets, we show that the typical nonexpansive mapping is a Rakotch contraction on every bounded subset and that there is a bounded set which is mapped into itself by this mapping. In particular, we obtain that the typical nonexpansive mapping in this setting has a unique fixed point which can be reached by iterating the mapping. Nevertheless, it turns out that the typical mapping is not a Rakotch contraction on the whole space and that it has the maximal possible Lipschitz constant of one on a residual subset of its domain. By typical we mean that the complement of the set of mappings with this property is $\sigma$-$\phi$-porous, that is, small in a metric sense. For a metric of pointwise convergence, we show that the set of Rakotch contractions is meagre.


\end{abstract}

\section{Introduction}

The question of existence of fixed points of nonexpansive mappings on bounded, closed and convex subsets of certain metric spaces has been studied from a number of different angles. Recall that a mapping $f\colon X\to X$ on a metric space $(X,\rho)$ is called \emph{nonexpansive} if it satisfies
\[
  \rho(f(x),f(y)) \leq \rho(x,y)
\]
for all $x,y\in X$. The starting point for these investigations is the classical fixed point theorem of Brouwer which states that every continuous self-mapping of a bounded, closed and convex subset of a Euclidean space has a fixed point. Without additional assumptions on the mapping, this theorem does not generalise to infinite dimensional Banach spaces. In this setting there are even Lipschitz mappings without a fixed point because in infinite dimensional Banach spaces the sphere is a Lipschitz retract of the ball; see~\cite{BS}. On the positive side, if we only consider nonexpansive mappings, then the fixed point theorem of Browder-Goehde-Kirk, see~\cite{Browder, GK1990, GR1984}, shows that under certain geometric assumptions on the Banach space every nonexpansive self-mapping of a bounded, closed and convex subset has a fixed point. On the other hand, there are concrete examples of nonexpansive mappings on bounded, closed and convex subsets of Banach spaces including, for instance, $\mathcal{C}([0,1])$, which lack a fixed point. In~\cite{BM1976} and~\cite{BM1989}, F. S.~de~Blasi and J.~Myjak showed that on every bounded, closed and convex subset of a Banach space the \emph{typical} nonexpansive mapping has a fixed point, which can be reached by iterating the mapping. Here the term typical means that the set of mappings with this property is large in a topological sense. More precisely, in the first of these articles it is shown that the set of mappings without a fixed point is a meagre subset of the space of all nonexpansive mappings and in the second one it is shown that this set is even $\sigma$-porous. Since Banach's fixed point theorem states that for a complete metric space $(X,\rho)$,
every \emph{strict contraction}, that is, every $f\colon X\to X$ satisfying
\[
  \rho(f(x),f(y)) \leq L \rho(x,y)
\]
for some $L<1$ and all $x,y\in X$, has a fixed point, it seems natural to ask whether this can serve as an explanation for the above phenomenon. Already in~\cite{BM1976},
de Blasi and Myjak showed that this is not the case by proving that in the Hilbert space setting the typical nonexpansive mapping is not a strict contraction. This result was later extended to Banach spaces in~\cite{BD2016} and to a large class of metric spaces in~\cite{BDR2017}. In~\cite{RZ2000, RZ2001} the second author together with A. J.~Zaslavski showed that the typical nonexpansive self-mapping of a bounded, closed and convex subset~$C$ of a Banach space is a \emph{Rakotch contraction}: in other words, for the typical nonexpansive $f\colon C\to C$, there is a decreasing function $\varphi_f\colon [0,\infty)\to [0,1]$, that is, $\varphi_f(s) \geq \varphi_f(t)$ for $s\leq t$, with $\varphi_f(t)<1$ for $t>0$ such that
\[
  \rho(f(x),f(y)) \leq \varphi_f(\rho(x,y)) \rho(x,y)
\]
for all $x,y\in C$. These mappings satisfy the assumptions of Rakotch's fixed point theorem given in~\cite{Rakotch} and hence have a fixed point which can be reached by sequences of iterates. An extension of this result to bounded, closed and convex subsets of hyperbolic spaces can be found in~\cite{MRZ2004, ReichESI, RZ2001} and to self-mappings of bounded, closed and convex subsets of positively curved spaces in~\cite{BDMR2021}.

In~\cite{Strobin}, F.~Strobin showed that on {\em unbounded} closed and convex subsets of Hilbert spaces the situation is dramatically different. Recall that the \emph{modulus of continuity} of a mapping $f$ is defined by
\[
  \omega_f(t) := \sup\{\rho(f(x),f(y))\colon x,y\in X, \; \rho(x,y) \leq t\}
\]
for $t\geq 0$. Given an increasing function $\omega\colon [0, \infty) \to [0,\infty)$, we denote by $C_\omega(X)$ the space of all continuous self-mappings of $X$ satisfying $\omega_f(t)\leq \omega(t)$ for all $t\geq 0$ equipped with the topology of uniform convergence on bounded subsets of~$X$. In~\cite{Strobin} Strobin showed that if $X$ is an unbounded closed and convex subset of a Hilbert space, then the subset of all mappings satisfying $\omega_f(t)<\omega(t)$ for some $t > 0$ is meagre. Setting $\omega(t)=t$, we see that this result implies, in particular, that the set of Rakotch contractions is a meagre subset of the space of nonexpansive mappings. Note that this fact is not just the negation of the statements of the results in the case of bounded domains, but also a negation in a stronger sense: it means that the set of Rakotch contractions is not only not the complement of a $\sigma$-porous set, but also a $\sigma$-porous set itself. In~\cite{RZ2015}, the second author together with A. J. ~Zaslavski showed that the typical self-mapping of an unbounded closed and convex subset of a hyperbolic space in the sense of~\cite{RS1990} is Rakotch contractive on bounded subsets. More precisely, they showed that the set without this property is a meagre set. Using a different metric, which generates a topology which is finer than the one of uniform convergence on bounded sets, in~\cite{RZ2016} they show that this set is even $\sigma$-porous with respect to this new metric. Since Strobin's results are obtained by using a variant of the Kirszbraun-Valentine extension theorem for uniformly continuous mappings, they seem to be tied to the Hilbert space setting. The aim of the current article is to extend Strobin's result on Rakotch contractions defined on an unbounded closed and convex subset of a Hilbert space to mappings on an unbounded complete  hyperbolic space, and to show that for two natural families of metrics, which generate the topology of uniform convergence on bounded sets, both the set of Rakotch contractions and the set of mappings which are not Rakotch contractive on bounded sets are $\sigma$-$\phi$-porous for suitable $\phi$. As a limiting case of one of these families, we retrieve the metric considered in~\cite{RZ2016}. In some sense the current article can be seen as an attempt to extend the results of~\cite{RZ2016} and~\cite{Strobin} to a common setting.

\section{Preliminaries and Notation}
For a metric space $(M,d)$, $z\in M$ and $r>0$, we use the notations
\[
  B(z,r):=\{x\in M\colon d(x,z)<r\}\qquad \text{and}\qquad \bar{B}(z,r) := \{x\in M\colon d(x,z)\leq r\}
\]
for open balls and closed balls, respectively. Since we are dealing with metric spaces of mappings defined on metric spaces, we sometimes want to stress the space in which the balls are defined. In this case we add in the sequel the name of the space or of the metric as a subscript. We call a mapping $\varphi\colon (a,b) \to \mathbb{R}$ \emph{increasing} if it satisfies $\varphi(s)\geq \varphi(t)$ whenever $s\geq t$. We call it \emph{strictly increasing} if $\varphi(s)>\varphi(t)$ whenever $s>t$.

\subsection{Hyperbolic metric spaces}
A \emph{geodesic metric space} is a metric space $(X,\rho)$ with the property that for all pairs of points $x,y\in X$ there is an isometry $c\colon [0,\rho(x,y)]\to X$ with $c(0)=x$ and $c(\rho(x,y))=y$. The image of such an isometry is sometimes called a \emph{metric segment connecting the points $x$ and $y$} or a \emph{geodesic}. If, in addition, the metric segments are unique, we call $(X,\rho)$ a \emph{uniquely geodesic space}. Note that in particular every convex subset of a Banach space is geodesic and every convex subset of a strictly convex Banach space is uniquely geodesic.
We consider geodesic spaces together with a family $\mathcal{S}$ of metric segments with the following properties:
\begin{enumerate}
\item For each pair of points $x,y\in X$ the family $\mathcal{S}$ contains a unique metric segment connecting $x$ and $y$. We denote this segment by $[x,y]$.
\item For every $x,y\in X$ and for all $z,w\in [x,y]$ we have $[z,w]\subset [x,y]$.
\end{enumerate}
In other words, for each pair $x,y\in X$ the family $\mathcal{S}$ contains a unique segment connecting them and it is closed under taking subsegments. For a uniquely geodesic space $X$, we choose $\mathcal{S}$ to be the family of all metric segments.

Let us mention an example of a metric space which is not uniquely geodesic but admits a natural choice of such a family~$\mathcal{S}$. Let $C\subset \ell_1$ be a closed and convex subset containing at least three non-colinear points. Then geodesics in $C$ are not unique but the family $\mathcal{S}$ consisting of all segments of the form
\begin{equation}\label{eq:SegmentConvComb}
  [x,y] := \{(1-\lambda) x + \lambda y\colon \lambda \in [0,1]\}, \qquad x,y\in C
\end{equation}
contains a unique segment for each pair of points and satisfies the above requirements. In the sequel we assume that given a geodesic metric space $X$, a fixed family~$\mathcal{S}$ is chosen and we do not mention it explicitly. The existence of unique metric segments allows for the definition of a convex combination of points in $X$, that is, given points $x,y\in X$ and a number $\lambda\in[0,1]$, we define the point
\[
  (1-\lambda)x \oplus \lambda y \in X
\]
as the unique element of $[x,y]$ satisfying
\begin{equation}\label{eq:HypLinComb}
  \rho(x,(1-\lambda)x \oplus \lambda y) = \lambda \rho(x,y) \qquad\text{and}\qquad
  \rho((1-\lambda)x \oplus \lambda y, y) = (1-\lambda) \rho(x,y).
\end{equation}
A triple $(X,\rho,\mathcal{S})$ where $(X,\rho)$ is a geodesic metric space and $\mathcal{S}$ is such a family of metric segments is called a \emph{hyperbolic space} if it satisfies
\begin{equation}\label{eq:HypIneq}
  \rho((1-\lambda) x\oplus \lambda y, (1-\lambda )x\oplus \lambda z) \leq \lambda \rho(y,z)
\end{equation}
for all $x,y,z\in X$ and all $\lambda\in [0,1]$.  Geometrically this means that triangles in~$X$ have to be thinner than Euclidean triangles with the same side lengths. Since the family $\mathcal{S}$ is closed under taking subsegments, we obtain
\begin{equation}\label{eq:HypLinComb2}
  \rho((1-\mu)x\oplus \mu y,(1-\lambda)x \oplus \lambda y) = |\lambda-\mu| \rho(x,y)
\end{equation}
from~\eqref{eq:HypLinComb} in the following way: assume w.l.o.g. that $\mu\leq \lambda$, and observe that $(1-\mu)x\oplus \mu y \in [x,(1-\lambda)x\oplus \lambda y]$ since this segment has to be a subset of $[x,y]$ and
\[
  (1-\mu)x\oplus \mu y = \left(1-\frac{\mu}{\lambda}\right) x \oplus \frac{\mu}{\lambda} ((1-\lambda) x \oplus \lambda y)
\]
which by~\eqref{eq:HypLinComb} implies that $\rho((1-\mu)x\oplus \mu y, (1-\lambda) x \oplus \lambda y) = (1-\frac{\mu}{\lambda})\lambda\rho(x,y) = |\lambda-\mu|\rho(x,y)$.

Note that this definition of hyperbolic spaces is slightly more general than the one of~\cite{RS1990} since we do note require that every pair of points be contained in a metric line, that is, an isometric copy of the real line~$\mathbb{R}$. It is also more general than the one of a $\mathrm{CAT}(0)$ space (see, for example, \cite[p.~158]{BH1999}) because we require the triangle condition only for points $(1-\lambda) x\oplus \lambda y$ and $(1-\lambda )x\oplus \lambda z$ and not for arbitrary points on the triangle. In particular, every closed and convex subset of a Banach space with the metric induced by the norm and the family of metric segments of the form~\eqref{eq:SegmentConvComb} is a hyperbolic metric space. A subset $C\subset X$ of a geodesic metric space $X$ is called \emph{convex} if for all $x,y\in C$ it contains the metric segment connecting these points, that is, $[x,y]\subset C$. We now check that in hyperbolic spaces balls are convex. For a hyperbolic space $X$ a point $z\in X$ and $R>0$ we consider $x,y\in B(z,R)$. For $\lambda\in(0,1)$ we use the triangle inequality, \eqref{eq:HypIneq} and~\eqref{eq:HypLinComb} to obtain
\begin{align*}
  \rho((1-\lambda)x\oplus\lambda y, z) &\leq \rho((1-\lambda)x\oplus\lambda y, (1-\lambda)x\oplus\lambda z) + \rho((1-\lambda)x\oplus\lambda z, z) \\ &\leq \lambda \rho(y,z) + (1-\lambda) \rho(x,z) < \lambda R + (1-\lambda) R =R
\end{align*}
which shows that $[x,y]\subset B(z,R)$. A similar argument works for closed balls. Note that given a hyperbolic space $(X,\rho,\mathcal{S})$ and a convex subset $C$, if we choose $\mathcal{S}_{C}$ as the family of all segments in $\mathcal{S}$ connecting points in $C$, then the triple $(C,\rho,\mathcal{S}_{C})$ is also a hyperbolic space.

The following lemma is in essence Lemma~2 of~\cite{IvesPreiss} where it is stated for convex subsets of Banach spaces.
\begin{lemma}\label{lem:LipschitzSubset}
  Let $X$ be a geodesic space, $Y$ a metric space and $\{Z_i\}_{i=1}^{\infty}$ a countable family of sets covering the space $X$. Let $f\colon X\to Y$ be a continuous mapping the restrictions of which to the sets $Z_i$ are Lipschitz and satisfy $\Lip(f|_{Z_i})\leq L$ for some $L>0$. Then $f$ is Lipschitz with $\Lip f \leq L$.
\end{lemma}

Although the following proof is essentially the one of Lemma~2 in~\cite{IvesPreiss}, we include it for the convenience of the reader and to keep the paper self-contained.

\begin{proof}
  Let $x,y\in X$ be arbitrary. We have to show that $\rho_Y(f(x),f(y))\leq L \rho_X(x,y)$. Since $X$ is geodesic, there is an isometry $c\colon[0,\rho_X(x,y)] \to X$ with $c(0)=x$ and $c(\rho_X(x,y))=y$. We set $a:=\rho_X(x,y)$ and consider the mapping $g\colon [0,a]\to Y$ with $g(t):= f(c(t))$. In order to show that $\rho_Y(f(x),f(y))\leq L \rho_X(x,y)$ we suppose to the contrary that $\rho_Y(g(0),g(a))=\rho_Y(f(x),f(y)) > L \rho_X(x,y)$. We set $M_i=c^{-1}(Z_i)$ and observe that $[0,a]=\bigcup_{i=1}^{\infty} M_i$. We set
  \[
    h\colon[0,a]\to\mathbb{R}, \qquad t \mapsto \rho_Y(g(0),g(t))-Lt
  \]
  and observe that, as the composition of continuous mappings, $h$ is a continuous function and by assumption satisfies $h(0)=0 < h(a)$. We set $S:=\{\sup(M_i)\colon i\in\mathbb{N}, \; M_i \neq \emptyset\}$. This is a countable subset of $[0,a]$ and hence $h(S)$ is countable too. Therefore we may pick $b\in [h(0),h(a))\setminus h(S)$ and use the intermediate value theorem and the observation that $h^{-1}(b)$ is a closed subset of the compact interval $[0,a]$ to obtain the largest $t\in[0,a]$ with $h(t)=b$. By this choice of $t$ we have $h(s)>h(t)$ for all $t<s\leq a$ and hence
  \[
    \rho_Y(g(s),g(t)) \geq \rho_Y(g(0),g(s)) - \rho_Y(g(0),g(t)) = h(s)-h(t) + L(s-t) > L(s-t).
  \]
  Pick $i\in\mathbb{N}$ with $t\in M_i$ and note that by assumption $\rho_Y(g(\tau),g(t)) \leq L \rho_X(\gamma(\tau),\gamma(t))=L |\tau-t|$ for all $\tau\in M_i$ and hence the above implies that $t= \max M_i$ and therefore $t\in S$, a contradiction to $h(t)=b\not\in h(S)$.
\end{proof}

\subsection{Notions of porosity}

Let $(M,d)$ be a metric space, $A\subset M$, $x\in A$, $\eta>0$ and $\phi\colon (0,\eta)\to (0,\infty)$ be an increasing function. Given $r>0$, we define
\[
  \gamma(x,r,A) := \sup \{s>0 \colon  \exists x' \in M \;\text{with}\; B_M(x',s)\subset B_M(x,r)\setminus A\},
\]
where we use the convention that $\sup\emptyset = - \infty$. The set $A$ is called \emph{$\phi$-lower porous at $x$} if
\[
  \liminf_{r\to 0} \frac{\phi(\gamma(x,r,A))}{r} > 0.
\]
A set $A$ is called $\phi$-lower porous if it is $\phi$-lower porous at all points $x\in A$. This notion has been introduced in~\cite[2.7]{Zajicek2} and in~Definition~2.1 in~\cite{Zajicek}. Since we do not use the notion of $\phi$-upper porosity, we refer to $\phi$-lower porous sets simply as $\phi$-porous sets. A set $A$ is called \emph{porous} if it is $\phi$-porous for $\phi(t)=t$ and \emph{$\sigma$-porous} if it is a countable union of porous sets. This notion of porosity is called lower porosity by L.~Zaj\'{\i}\v{c}ek in~\cite{Zajicek}. We will heavily use the following characterisation of $\phi$-porous sets due to M.~Dymond.

\begin{lemma}[Lemma~2.2 in~\cite{Dymond2021}]\label{lem:PorousMichael}
  Let $(M,d)$ be a metric space without isolated points, $A\subset M$, $x\in A$, $\eta>0$ and
  \[
    \phi\colon (0,\eta) \to (0,\infty)
  \]
  be a strictly increasing, concave function with $\lim_{t\to 0} \phi(t)=0$. Then $A$ is $\phi$-porous at $x$ if and only if there are an $r_0>0$ and an $\alpha\in(0,1)$ such that for every $r\in (0,r_0)$ there is an $x'\in M$ such that $0<d(x,x')\leq r$ and $B_M(x',\phi^{-1}(\alpha r)) \cap A = \emptyset$.
\end{lemma}

In the case of $\phi(t)=t$ this lemma recovers the classical definition of porous sets, that is, a subset $A$ of a metric space $(M,d)$ is called \emph{porous at $x\in A$} if there are $r_0>0$ and $\alpha>0$ such that for every $r\in (0,r_0)$, there is a point $y\in M$ with $d(x,y)<r$ and $B(y,\alpha r)\cap A=\emptyset$. The set $A$ is called \emph{porous} if it is porous at all its elements.

For a detailed discussion of different variants of porosity, we refer the interested reader to~\cite{Zajicek}. Note that although our definition of porous sets differs from the one in some of the articles mentioned in the introduction, for $\sigma$-porous sets these definitions agree. A detailed discussion of the relationships among these notions of porosity can be found in the introduction of~\cite{BDR2017}. Note that while all variants of porosity are metric properties, they imply that the set is meagre and hence also small in a topological sense.

\subsection{Two classes of metrics generating the topology of uniform convergence on bounded sets}
For an unbounded complete hyperbolic metric space $X$ we consider the space
\[
  \mathcal{M} := \{f\colon X\to X\colon \Lip f \leq 1\}
\]
of nonexpansive self-mappings. We present two classes of metrics generating the topology of uniform convergence on bounded subsets of~$X$. One of these is defined as a series and the other one as a weighted variant of the metric of uniform convergence. The first metric is a generalisation of the metric
\[
  d_{\theta}(f,g) = \sum_{n=1}^{\infty} 2^{-n} \frac{d_{n,\theta}(f,g)}{1+d_{n,\theta}(f,g)}
\]
where we fix $\theta\in X$ and set
\[
  d_{n,\theta}(f,g) = \sup\{\rho(f(x),f(y))\colon x,y\in \bar{B}(\theta,n)\}.
\]
This type of metric can be seen as a variant of a natural metric on infinite products of metric spaces. Instead of the quotient $\frac{d_{n,\theta}(f,g)}{1+d_{n,\theta}(f,g)}$ it is also possible to use $\min\{d_{n,\theta}(f,g),1\}$. Metrics of this type are used for example in~\cite[p.~323]{Engelking} and~\cite{Rolewicz}.

The other family of metrics we use here are the metrics
\[
  d_{\theta,s}(f,g) = \sup_{x\in X}\frac{\rho(f(x),g(x))}{1+\rho(x,\theta)^s}
\]
for $s > 1$ and some fixed $\theta\in X$. In the limiting case $s=1$ this is the metric used in the article~\cite{RZ2016}. In Proposition~\ref{prop:dsthetaTop} we show that for $s>1$ this metric generates the topology of uniform convergence on bounded sets whereas for $s=1$ it does not.

Note that in the case of a bounded space $X$ both of these metrics are equivalent to the metric of uniform convergence $d_{\infty}(f,g):=\sup_{x\in X} \{\rho(f(x),g(x)\}$: 
We set $C:=\diam X$ and $N := \lceil C\rceil$. For the first metric observe that $d_{n,\theta}(f,g) = d_\infty(f,g)$ for all $n\geq N$ and $d_{n,\theta}(f,g) \leq d_{\infty}(f,g)$ for $n\leq N$. Since $t\mapsto \frac{t}{1+t}$ is an increasing function this implies that
\[
  \frac{1}{1+C}\Big(\sum_{n=N}^{\infty} 2^{-n}\Big) d_{\infty}(f,g) \leq \sum_{n=1}^{\infty} 2^{-n} \frac{d_{n,\theta}(f,g)}{1+d_{n,\theta}(f,g)} \leq d_{\infty}(f,g).
\]
For the metric $d_{\theta,s}$, since $1 \leq 1+\rho(x,\theta)^s \leq 1 + C^s$, we have
\[
  \frac{1}{1+C^s} d_{\infty}(f,g) \leq d_{\theta,s}(f,g) \leq  d_{\infty}(f,g).
\]
This shows that the metrics $d_\theta$, $d_{\theta,s}$ and $d_\infty$ are equivalent when $X$ is bounded.

\subsubsection{Metrics defined by a series}
Let $I$ be one of the intervals $(0,1)$ or $(0,1]$ and let
\[
  \phi \colon I \to (0,\infty)
\]
be a continuous strictly increasing function with the following properties.
\begin{enumerate}
\item[(C1)] \label{C1} There is an $\eta_\phi\in I$ such that $\phi$ is concave on $(0,\eta_\phi)$ and $\phi(\eta_{\phi})\geq\eta_\phi$.
\item[(C2)] \label{C2} $\displaystyle \lim_{t\to 0} \phi(t) = 0$ and there is a $t\in I$ with $\phi(t)=1$.
\item[(C3)] \label{C3}For every $k\in\mathbb{N}$ there is a $C_k>0$ such that
  \[
    \phi^{-1}\left(\frac{1}{m}\right) \leq C_k \phi^{-1}\left(\frac{1}{k+m}\right)
  \]
  for all $m\in\mathbb{N}$.
\item[(C4)] \label{C4} The series $\displaystyle\sum_{n=1}^{\infty} n \phi^{-1}\left(\frac1n\right)$ converges. We denote by $C_{\phi}$ the maximum of its sum and one.
\end{enumerate}

\begin{remark}\label{rem:PhiInvConv}
  Note that since $\phi$ is increasing and concave on $(0,\eta_\phi)$, its inverse is also increasing and convex on~$(0,\phi(\eta_\phi))$. In particular, (C1) and (C2) imply that
  \begin{equation}\label{eq:PhiInvConvIneq}
    \phi^{-1}(at) \leq t \phi^{-1}(a)
  \end{equation}
  for $t\in (0,1)$ and $a\in (0,\phi(\eta_\phi)]$. In particular, we obtain for $t\in(0,\phi(\eta_\phi))$ that
  \begin{equation}\label{eq:PhiInvSmT}
    \phi^{-1}(t) = \phi^{-1}\left(\frac{t}{\phi(\eta_\phi)} \phi(\eta_\phi)\right) \leq \frac{t}{\phi(\eta_\phi)} \phi^{-1}(\phi(\eta_\phi)) = t \frac{\eta_{\phi}}{\phi(\eta_\phi)} \leq t
  \end{equation}
  since $\phi(\eta_\phi)\geq\eta_\phi$.
\end{remark}

Since a non-singleton geodesic metric space cannot contain isolated points, conditions~(C1) and~(C2) allow us to use Lemma~\ref{lem:PorousMichael} to check $\phi$-porosity for functions~$\phi$ satisfying the above conditions. Since the square root function is also strictly increasing, concave, continuous and satisfies $\sqrt{0}=0$, also $\sqrt{\phi}$ satisfies the conditions of Lemma~\ref{lem:PorousMichael}.

\begin{example}\label{ex:Exp}
  We now check that the function
  \[
    \phi\colon (0,1)\to (0,\infty), \qquad t \mapsto -\frac{1}{\log_2 t},
  \]
  where $\log_2$ denotes the dyadic logarithm, satisfies the above conditions with $\eta_\phi=\mathrm{e}^{-2}$. Since $\log_2 t$ is continuous, increasing and does not vanish on $(0,1)$, also the function $\phi$ is continuous and increasing. Since $\lim_{t\to 0} \log_2 t = -\infty$ and $\phi(\frac{1}{2})=1$ it satisfies~(C2). Since for $t<\mathrm{e}^{-2}$ we have
  \[
    \phi''(t) = -\log 2\frac{2+ \log t}{t^2(\log t)^3} < 0
  \]
  it also satisfies~(C1) with $\eta_\phi=\mathrm{e}^{-2}$. Since $\phi$ is strictly increasing and
  \[
    \phi(2^{-1/t}) = - \frac{1}{\log_22^{-1/t}} = t
  \]
  we have $\phi^{-1}(t) = 2^{-1/t}$. This immediately implies~(C4). In order to show that $\phi$ satisfies~(C3) for $k,m\in\mathbb{N}$, we set $C_k:=2^k$ and note that
  \[
    \phi^{-1}\left(\frac{1}{k+m}\right) = 2^{-(k+m)} = 2^{-k} \phi^{-1}\left(\frac{1}{m}\right)
  \]
  which implies in particular that
  \[
    \phi^{-1}\left(\frac{1}{m}\right) \leq C_k \phi^{-1}\left(\frac{1}{k+m}\right)
  \]
  as required.
\end{example}

\begin{example}
  We now check that the function
  \[
    \phi\colon (0,1] \to (0,\infty), \qquad t \mapsto t^{1/4}
  \]
  satisfies the above conditions. Since it is obviously continuous, increasing and concave, we only have to check conditions~(C2)--(C4). Since $\phi(1)=1$ and $\sqrt[4]{0}=0$, it satisfies~(C2). Since~$\phi^{-1}(t)=t^4$ it satisfies~(C4) and by
  \begin{align*}
    \phi^{-1}\left(\frac{1}{k+m}\right) & = \frac{1}{(k+m)^4} = \frac{1}{\sum_{j=0}^{4}\binom{4}{j} k^jm^{4-j}} = \frac{1}{m^4} \frac{1}{\sum_{j=0}^{4}\binom{4}{j} k^jm^{-j}} \\ &\geq \frac{1}{\sum_{j=0}^{4}\binom{4}{j} k^j} \phi^{-1}\left(\frac{1}{m}\right)
  \end{align*}
  it also satisfies~(C3) with
  \[
    C_k := \sum_{j=0}^{4}\binom{4}{j} k^j.
  \]
  Hence $\phi$ satisfies all the requirements.
\end{example}

We choose a $\theta\in X$, set
\[
  d_{n,\theta}(f,g) = \sup \{ \rho(f(x),g(x))\colon x\in \bar{B}(\theta, n)\}
\]
and define the metric $d_{\theta, \phi}$ by
\begin{equation}\label{eq:Metric}
  d_{\theta,\phi}(f,g) := \sum_{n=1}^{\infty} \phi^{-1}\left(\frac{1}{n}\right) \frac{d_{n,\theta}(f,g)}{1+d_{n,\theta}(f,g)}.
\end{equation}
Observe that it is well defined since
\[
  \sum_{n=1}^{\infty} \phi^{-1}\left(\frac{1}{n}\right) \frac{d_{n,\theta}(f,g)}{1+d_{n,\theta}(f,g)} \leq \sum_{n=1}^{\infty} \phi^{-1}\left(\frac{1}{n}\right) \leq \sum_{n=1}^{\infty} n \phi^{-1}\left(\frac{1}{n}\right) < \infty
\]
by~(C4). The symmetry of the expression $d_{n,\theta}(f,g)$ in $f$ and $g$ implies that $d_{\theta,\phi}$ is symmetric and the triangle inequality for $d_{\theta,\phi}$ follows from the triangle inequality for $d_{n,\theta}$ and the observation that the sum of two non-negative numbers is bounded from below by both summands. This shows that $d_{\theta,\phi}$ is indeed a metric. Moreover, given $\theta_1\neq \theta_2$, the metrics $d_{\theta_1,\phi}$ and $d_{\theta_2,\phi}$ are equivalent. In order to show this, we pick a natural number $k\in\mathbb{N}$ with $\rho(\theta_1,\theta_2)\leq k$ and observe that the triangle inequality implies that
\[
  \bar{B}(\theta_1,n) \subset \bar{B}(\theta_2,n+k) \qquad \text{and} \qquad \bar{B}(\theta_2,n) \subset \bar{B}(\theta_1,n+k)
\]
for all $n\in\mathbb{N}$. Hence, we have
\[
  d_{n,\theta_1}(f,g) \leq d_{n+k,\theta_2}(f,g) \qquad \text{and} \qquad d_{n,\theta_2}(f,g)\leq d_{n+k,\theta_1}(f,g)
\]
which, by~(C3) and the fact that $t\mapsto \frac{t}{1+t}$ is an increasing function, imply that
\begin{multline*}
  \sum_{n=1}^{\infty} \phi^{-1}\left(\frac{1}{n}\right) \frac{d_{n,\theta_2}(f,g)}{1+d_{n,\theta_2}(f,g)} \leq \sum_{n=1}^{\infty} \phi^{-1}\left(\frac{1}{n}\right) \frac{d_{n+k,\theta_1}(f,g)}{1+d_{n+k,\theta_1}(f,g)}\\ \leq C_k \sum_{n=1}^{\infty} \phi^{-1}\left(\frac{1}{n+k}\right) \frac{d_{n+k,\theta_1}(f,g)}{1+d_{n+k,\theta_1}(f,g)} \leq C_k \sum_{m=1}^{\infty} \phi^{-1}\left(\frac{1}{m}\right) \frac{d_{m,\theta_1}(f,g)}{1+d_{m,\theta_1}(f,g)}.
\end{multline*}
Since the situation is symmetric with respect to $\theta_1$ and $\theta_2$, the equivalence of the metrics follows.

The following lemma, which links the pointwise distance between two mappings to the distance in the metric of uniform convergence on bounded sets, turns out be very useful in our subsequent considerations.

\begin{lemma}\label{lem:localAndGlobalDistance}%
  Let $f,g\in\mathcal{M}$ with $d_{\theta,\phi}(f,g) \leq \frac{1}{2}\phi^{-1}(\tfrac{1}{m})$. Then we have
  \[
    \rho(f(z),g(z)) \leq \frac{2}{\phi^{-1}(\tfrac{1}{m})} d_{\theta,\phi}(f,g)
  \]
  for all $z\in \bar{B}(\theta,m)$. In particular, the inequality $d_{\theta,\phi}(f,g) \leq \frac{r}{2} \phi^{-1}(\tfrac{1}{m})$ for some $r\in(0,1]$ implies that $\rho(f(z),g(z)) \leq r$ for  all $z\in \bar{B}(\theta,m)$.  If $m\geq \frac{1}{\phi(\eta_\phi)}$ too, then the inequality $d_{\theta,\phi}(f,g) \leq r \phi^{-1}(\tfrac{1}{2m})$ for some $r\in(0,1]$ implies that $\rho(f(z),g(z)) \leq r$ for all $z\in\bar{B}(\theta,m)$.
\end{lemma}

\begin{proof}
  By definition, for every $m\in\mathbb{N}$, we have
  \[
    d_{\theta,\phi}(f,g) = \sum_{n=1}^{\infty} \phi^{-1}\left(\tfrac{1}{n}\right) \frac{d_{\theta,n}(f, g)}{1+d_{\theta,n}(f, g)} \geq \phi^{-1}\left(\tfrac{1}{m}\right) \frac{d_{\theta,m}(f, g)}{1+d_{\theta,m}(f, g)}.
  \]
  Using the observation that
  \[
    \frac{a}{1+a} \leq \frac{1}{2} \Leftrightarrow a \leq \frac{1}{2} + \frac{a}{2} \Leftrightarrow a \leq 1,
  \]
  we conclude that $d_{\theta,m}(f, g) \leq 1$ whenever $d_{\theta,\phi}(f, g) \leq \frac{1}{2} \phi^{-1}(\frac{1}{m})$ and hence
  \[
    \frac{d_{\theta,m}(f,g)}{2} \leq \frac{d_{\theta,m}(f,g)}{1+d_{\theta,m}(f,g)} \leq \frac{d_{\theta,\phi}(f,g)}{\phi^{-1}(\tfrac{1}{m})}
  \]
  in this case. Combining these observations, we see that
  \[
    \rho(f(z),g(z)) \leq d_{\theta,m}(f, g) \leq \frac{2}{\phi^{-1}(\tfrac{1}{m})} d_{\theta,\phi}(f, g)
  \]
  for all $z\in \bar{B}(\theta,m)$ provided that $d_{\theta,\phi}(f, g) \leq
  \frac{1}{2} \phi^{-1}(\tfrac{1}{m})$. The claim for the case where $m\geq\frac{1}{\phi(\eta_\phi)}$ follows from the above inequality together with Remark~\ref{rem:PhiInvConv} which shows that
  \[
    \phi^{-1}(\tfrac{1}{2m}) \leq \tfrac{1}{2}\phi^{-1}(\tfrac{1}{m})
  \]
  for $m\geq \frac{1}{\phi(\eta_\phi)}$.
\end{proof}

\begin{proposition}
  The metric $d_{\theta,\phi}$ generates the topology of uniform convergence on bounded sets.
\end{proposition}

\begin{proof}
  Since the topology of uniform convergence on bounded sets is first-countable it is enough to work with sequences.
  Note that Lemma~\ref{lem:localAndGlobalDistance} implies that convergence with respect to $d_{\theta,\phi}$ implies uniform convergence on bounded sets. For the converse implication assume that $f_k\to f$ uniformly on bounded sets and let $\varepsilon>0$ be given. Using~(C4) we may pick an $N\in\mathbb{N}$ such that
  \[
    \sum_{n=N+1}^{\infty} \phi^{-1}\left(\tfrac{1}{n}\right) \frac{d_{\theta,n}(f_k,f)}{1+d_{\theta,n}(f_k,f)} \leq \sum_{n=N+1}^{\infty} \phi^{-1}\left(\tfrac{1}{n}\right) < \frac{\varepsilon}{2}
  \]
  for every $k\in\mathbb{N}$. Since $f_k\to f$ uniformly on bounded sets, we may choose a $K\in\mathbb{N}$ such that
  \[
    d_{\theta,n}(f_k,f) < \frac{\varepsilon}{2C_\phi}
  \]
  for all $n=1,\ldots, N$ and all $k\geq K$. Combining these inequalities, we arrive at
  \begin{align*}
    d_{\theta,\phi}(f_k,f) &= \sum_{n=1}^{N} \phi^{-1}\left(\tfrac{1}{n}\right) \frac{d_{\theta,n}(f_k,f)}{1+d_{\theta,n}(f_k,f)} + \sum_{n=N+1}^{\infty} \phi^{-1}\left(\tfrac{1}{n}\right) \frac{d_{\theta,n}(f_k,f)}{1+d_{\theta,n}(f_k,f)} \\& < \sum_{n=1}^{N} \phi^{-1}\left(\tfrac{1}{n}\right) d_{\theta,n}(f_k,f) + \frac{\varepsilon}{2} \leq \varepsilon
  \end{align*}
  for all $k\geq K$, which shows that $f_k\to f$ in $(\mathcal{M},d_{\theta,\phi})$.
\end{proof}

\begin{proposition}\label{prop:StrContrDense}
  The set of strict contractions is a dense subspace of $(\mathcal{M},d_{\theta,\phi})$. More precisely, given $f\in\mathcal{M}$ and $\gamma\in(0,1)$ the mapping
  \[
    f_\gamma(x) := (1-\gamma) f(x) \oplus  \gamma f(\theta)
  \]
  satisfies $\Lip f_\gamma \leq 1 - \gamma$ and $d_{\theta,\phi}(f,f_\gamma) \leq C_{\phi} \gamma$.
\end{proposition}

\begin{proof}
  For $x,y\in X$ we have
  \begin{align*}
    \rho(f_\gamma(x), f_\gamma(y)) & = \rho((1-\gamma) f(x) \oplus \gamma f(\theta),  (1-\gamma) f(y) \oplus\gamma f(\theta)) \leq (1-\gamma) \rho(f(x),f(y))\\ & \leq (1-\gamma) \rho(x,y)
  \end{align*}
  by~\eqref{eq:HypIneq} and hence $\Lip f_\gamma \leq 1-\gamma$.

  For $n\in\mathbb{N}$ and $x\in \bar{B}(\theta,n)$ we have
  \[
    \rho(f(x),f_\gamma(x)) = \rho(f(x), (1-\gamma) f(x) \oplus \gamma f(\theta)) = \gamma \rho(f(\theta),f(x)) \leq \gamma\rho(\theta,x) \leq n \gamma
  \]
  by~\eqref{eq:HypLinComb} and hence
  \begin{align*}
    d_{\theta,\phi}(f,f_\gamma) & = \sum_{n=1}^{\infty} \phi^{-1}(\tfrac{1}{n}) \frac{d_{\theta,n}(f,f_{\gamma})}{1+d_{\theta,n}(f,f_{\gamma})} \leq \sum_{n=1}^{\infty} \phi^{-1}(\tfrac{1}{n}) \frac{n\gamma}{1+n\gamma}\leq \gamma \sum_{n=1}^{\infty}n \phi^{-1}(\tfrac{1}{n}) \leq \gamma C_{\phi}
  \end{align*}
  as claimed.
\end{proof}

\begin{remark}
  Note that a similar result for the metric of uniform convergence cannot be achieved because for this metric, approximating the identity with strict contractions is not possible. In other words, every nonexpansive mapping which is sufficiently close to the identity has to have Lipschitz constant one.
\end{remark}

\subsubsection{Metrics of weighted uniform convergence}
We fix $s\geq 1$ and $\theta\in X$ and define the metric
\[
  d_{\theta,s}(f,g) := \sup_{x\in X} \frac{\rho(f(x),g(x))}{1+\rho(x,\theta)^s}
\]
which is well defined since
\[
  \frac{\rho(f(x),g(x))}{1+\rho(x,\theta)^s} \leq \frac{2 \rho(x,\theta) + \rho(f(\theta), g(\theta))}{1+\rho(x,\theta)^s} \leq \rho(f(\theta), g(\theta)) + 2
\]
which follows from the triangle inequality because $f$ and $g$ are nonexpansive. The case of $s=1$ is the metric considered in~\cite{RZ2016} which by Proposition~\ref{prop:dsthetaTop} generates a topology which is strictly finer than the one of uniform convergence on bounded sets. Together with this metric, we also consider the function
\begin{equation}\label{eq:PsiS}
  \psi_s \colon (0,\infty) \to (0,\infty), \qquad t\mapsto t^{1/s},
\end{equation}
which is strictly increasing, concave and satisfies $\lim_{t\to 0} \psi_s(t)=0$. In particular we are able to check $\psi_s$-porosity using Lemma~\ref{lem:PorousMichael}.

\begin{proposition}
  Let $s\geq 1$. For $\theta_1, \theta_2\in X$ the metrics $d_{\theta_1,s}$ and  $d_{\theta_2,s}$ are equivalent.
\end{proposition}

\begin{proof}
  Using the triangle inequality and the convexity of $t\mapsto t^s$ we obtain
  \begin{align*}
    1 + \rho(x,\theta_1)^s &\leq 1 + (\rho(x,\theta_2)+\rho(\theta_1,\theta_2))^s \leq 1 + 2^{s-1} ( \rho(x,\theta_2)^s + \rho(\theta_1,\theta_2)^s )\\
                           &\leq 2^{s}(1+\rho(x,\theta_2)^s) + 2^s\rho(\theta_1,\theta_2)^s (1+\rho(x,\theta_2)^s) \leq 2^{s} (1+ \rho(\theta_1,\theta_2)^s) \left(1+\rho(x, \theta_2)^s\right)
  \end{align*}
  since $\max\{1,2^{s-1}\}\leq 2^s$ and deduce that
  \[
    d_{\theta_1,s}(f,g) \geq \frac{1}{2^{s}(1+\rho(\theta_1,\theta_2)^s)} d_{\theta_2,s}(f,g).
  \]
  Since the above situation is symmetric with respect to $\theta_1$ and $\theta_2$, we conclude that the metrics $d_{\theta_1,s}$ and $d_{\theta_2,s}$ are equivalent.
\end{proof}

Also for these metrics we have results similar to the ones for the metric $d_{\theta,\phi}$.

\begin{lemma}\label{lem:sLocalToGlobal}
  For $f,g\in\mathcal{M}$ we have
  \[
    \rho(f(x),g(x)) \leq (1+\rho(x,\theta)^s) d_{\theta,s}(f,g)
  \]
  for $x\in X$.
\end{lemma}

\begin{proof}
  We have
  \begin{align*}
    \rho(f(x),g(x)) = \frac{\rho(f(x),g(x))}{1+\rho(x,\theta)^s} (1+ \rho(x,\theta)^s) \leq d_{\theta,s}(f,g) (1+ \rho(x,\theta)^s)
  \end{align*}
  as claimed.
\end{proof}

\begin{proposition}\label{prop:StrictContrDenseFords}
  The set of strict contractions is a dense subspace of $(\mathcal{M},d_{\theta,s})$. More precisely, given $f\in\mathcal{M}$ and $\gamma\in(0,1)$ the mapping
  \[
    f_\gamma(x) :=  (1-\gamma) f(x) \oplus \gamma f(\theta)
  \]
  satisfies $\Lip f_\gamma \leq 1 - \gamma$ and $d_{\theta,s}(f,f_\gamma) \leq \gamma$.
\end{proposition}

\begin{proof}
  Since Proposition~\ref{prop:StrContrDense} already yields the bound on the Lipschitz constant, we only have to show that $d_{\theta,s}(f,f_\gamma) \leq \gamma$ which follows, in its turn, from
  \begin{align*}
    \frac{\rho(f(x),f_{\gamma}(x))}{1+\rho(x,\theta)^s} &= \frac{\rho(f(x), (1-\gamma) f(x) \oplus  \gamma f(\theta))}{1+\rho(\theta,x)^s} = \frac{\gamma \rho(f(x),f(\theta))}{1+\rho(x,\theta)^s} \leq \gamma \frac{\rho(x,\theta)}{1+\rho(x,\theta)^s} \leq \gamma.
  \end{align*}
\end{proof}

\begin{proposition}\label{prop:dsthetaTop}
  The metric $d_{\theta,s}$ generates the topology of uniform convergence on bounded sets if and only if $s>1$. For $s=1$ it generates a topology which is stronger than the one of uniform convergence.
\end{proposition}

\begin{proof}
  Let $s\geq1$. Since the topology of uniform convergence on bounded sets is first-countable it is enough to work with sequences. First note that Lemma~\ref{lem:sLocalToGlobal} implies that every sequence of functions which converges for $d_{\theta,s}$ converges uniformly on bounded sets. For the converse implication in the case of $s>1$ let $f\in\mathcal{M}$ and $\varepsilon>0$ be given. We choose $n > (s-1)^{-1/s}$ large enough such that $\frac{2n}{1+n^s}<\frac{\varepsilon}{3}$. Note that the function $t\mapsto \frac{t}{1+t^s}$ is decreasing for $t>(s-1)^{-1/s}$ since its derivative satisfies
  \[
    \frac{1-(s-1)t^s}{(1+t^s)^2} < 0
  \]
  for these $t$. For all $g\in\mathcal{M}$ with $\rho(f(x),g(x)) < \frac{\varepsilon}{3}$ for $x\in B(\theta,n)$ we have
  \begin{align*}
    d_{\theta,s}(f,g) & \leq \sup_{x\in B(\theta,n)} \frac{\rho(f(x),g(x))}{1+\rho(x,\theta)^s} + \sup_{\rho(x,\theta) \geq n } \frac{\rho(f(x),g(x))}{1+\rho(x,\theta)^s}\\ & \leq \frac{\varepsilon}{3} + \sup_{\rho(x,\theta) \geq n } \frac{\rho(f(\theta),g(\theta))+2\rho(x,\theta)}{1+\rho(x,\theta)^s}  \leq \frac{2\varepsilon}{3} + \frac{2n}{1+n^s} < \varepsilon.
  \end{align*}
  In other words, we have exhibited a neighbourhood of $f$ for the topology of uniform convergence on bounded sets which is contained in the ball $B_{d_{\theta,s}}(f,\varepsilon)$.

  In the case of $s=1$, we construct a sequence of mappings which converges uniformly on bounded sets but not for the metric~$d_{\theta,s}$. We pick a sequence $(x_n)_{n\in\mathbb{N}}$ with $\rho(\theta,x_n)=2n$ and define
  \[
    \lambda_n(x) := \frac{1}{2} \max\left\{1-\frac{\rho(x,x_n)}{n},0\right\}\qquad\text{and}\qquad f_n(x) := (1-\lambda_n(x)) \theta \oplus \lambda_n(x) x_n.
  \]
  Note that
  \[
    |\lambda_n(x)-\lambda_n(y)| = \frac{1}{2n} |\rho(x,x_n)-\rho(y,x_n)| \leq \frac{1}{2n} \rho(x,y)
  \]
  for $x,y,\in B(x_n,n)$ and hence $\Lip \lambda_n \leq \frac{1}{2n}$ by Lemma~\ref{lem:LipschitzSubset} since $\lambda_n(x)=0$ outside of this ball. Hence by~\eqref{eq:HypLinComb2} we obtain
  \begin{align*}
    \rho(f_n(x),f_n(y)) & = \rho((1-\lambda_n(x)) \theta \oplus \lambda_n(x) x_n, (1-\lambda_n(y)) \theta \oplus \lambda_n(y) x_n)\\ &= |\lambda_n(x)-\lambda_n(y)| \rho(\theta,x_n) \leq \rho(x,y)
  \end{align*}
  which shows that $f_n$ is nonexpansive. Moreover, note that for $x\in B(\theta,n)$ we have
  \[
    \rho(x_n,x) \geq \rho(x_n,\theta)-\rho(x,\theta) \geq n
  \]
  and hence $\lambda_n(x)=0$. In other words $f_n(x)=\theta$ for all $x\in B(\theta,n)$. This shows that $f_n\to g$, where $g(x)=\theta$ for all $x\in X$, uniformly on bounded sets. On the other hand, we have
  \[
    f_n(x_n)=\frac{1}{2}\theta\oplus\frac{1}{2}x_n
  \]
  and hence
  \[
    \sup_{x\in X} \frac{\rho(f_n(x),g(x))}{1+\rho(x,\theta)} \geq \frac{\rho(f_n(x_n),g(x_n))}{1+\rho(x_n,\theta)} = \frac{n}{1+2n} \geq \frac{1}{3}
  \]
  which shows that the sequence does not converge to $g$ for the metric $d_{\theta,1}$.
\end{proof}

\section{Is the generic mapping a Rakotch contraction?}\label{sec:Rakotch}

The aim of this section is to investigate the size of the set of mappings which are Rakotch contractions or whose restrictions to suitable bounded sets are Rakotch contractions. We start by exhibiting a large set of mappings which map suitably large balls into themselves.

\begin{theorem}\label{thm:BallIntoBall}
  There exists a set $\mathcal{F}\subset\mathcal{M}$ such that $\mathcal{M}\setminus\mathcal{F}$ is $\sqrt{\phi}$-porous in $(\mathcal{M}, d_{\theta,\phi})$ and every $f\in\mathcal{F}$ has the following property:
  \begin{equation}\tag{P1}\label{eq:P1}
    \exists M_f > 0 \;\text{such that}\; f(\bar{B}(\theta,M_f)) \subset \bar{B}(\theta,M_f)
  \end{equation}
  If $\phi$ satisfies the condition
  \begin{equation}\label{eq:C5}\tag{C5}
    \phi^{-1}\left(\tfrac{t}{a+b}\right) \leq \phi^{-1}\left(\tfrac{t}{a}\right)\phi^{-1}\left(\tfrac{t}{b}\right)
  \end{equation}
  for $a,b\geq 1$ and $t\in(0,1)$, then $\mathcal{M}\setminus\mathcal{F}$ is $\phi$-porous in $(\mathcal{M}, d_{\theta,\phi})$.
\end{theorem}

\begin{remark}
  An example of a~$\phi$ satisfying condition~\eqref{eq:C5} is the function
  \[
    \phi\colon (0,1)\to (0,\infty), \qquad t \mapsto -\frac{1}{\log_2 t},
  \]
  considered in Example~\ref{ex:Exp} because
  \[
    \phi^{-1}\left(\tfrac{t}{a+b}\right) = 2^{-\frac{a+b}{t}} = 2^{-a/t}2^{-b/t} = \phi^{-1}\left(\tfrac{t}{a}\right) \phi^{-1}\left(\tfrac{t}{b}\right).
  \]
  In this case~\eqref{eq:C5} is satisfied even as an equality.
\end{remark}

We also have a similar result for the metric $d_{\theta,s}$.

\begin{theorem}\label{thm:BallIntoBall2}
  There exists a set $\mathcal{F}\subset\mathcal{M}$ such that $\mathcal{M}\setminus\mathcal{F}$ is $\psi_s$-porous in $(\mathcal{M}, d_{\theta,s})$ and every $f\in\mathcal{F}$ has property~\eqref{eq:P1}.
\end{theorem}

The proofs of these theorems are given in the next section.

\begin{remark}\label{rem:Balls}
  Property~\eqref{eq:P1} not only implies the existence of a ball which is mapped into itself but the seemingly stronger statement that all large enough balls around $\theta$ are mapped into themselves. In order to see this, let $f\in\mathcal{M}$ be a nonexpansive mapping satisfying (P1) with some $M_f>0$. Now for every $M>M_f$ and every $x\in \bar{B}(\theta,M)$ with $\rho(\theta,x)\geq M_f$, we may pick a point $z\in [\theta,x]$ with $\rho(z,\theta)=M_f$ and $\rho(x,z)\leq M-M_f$. Using the triangle inequality and the fact that $f$ satisfies~(P1) with $M_f$ we observe that
  \[
    \rho(f(x),\theta) \leq \rho(f(x),f(z)) + \rho(f(z),\theta) \leq M - M_f + M_f = M.
  \]
  If $\rho(x,\theta) \leq M_f$, property~\eqref{eq:P1} directly implies that $f(x)\in \bar{B}(\theta,M)$. Hence the ball $\bar{B}(\theta,M)$ is mapped into itself by $f$.
\end{remark}

These theorems show that in both cases for the typical nonexpansive mapping, all large enough balls around~$\theta$ are mapped into themselves.

We now show that, in addition, the typical nonexpansive self-mapping is a Rakotch contraction on every bounded subset. More precisely, we show that the typical mapping has the following property:
\begin{itemize}
\item[(R)\namedlabel{it:RB}{R}] For every bounded $B\subset X$ there is a decreasing  function $\varphi_{f,B}\colon (0,\infty)\to (0,1)$ such that
\[
  \rho(f(x),f(y)) \leq \varphi_{f,B}(\rho(x,y)) \rho(x,y)
\]
for all $x,y\in X$ with $x\neq y$.
\end{itemize}
To this end, we set
\[
  \mathcal{G}_n := \left\{f \in \mathcal{M} \colon \sup \left\{\frac{\rho(f(x),f(y))}{\rho(x,y)} \colon x,y \in \bar{B}(\theta,n) \;\text{and}\; \rho(x,y) \geq \frac{1}{n}\right\}<1\right\}
\]
for $n\in\mathbb{N}$.

\begin{theorem}\label{thm:RakotchOnBoundedSubsetsIsSigmaPorous}
  For each $n\in\mathbb{N}$ the set $\mathcal{M}\setminus\mathcal{G}_n$ is porous in $(\mathcal{M}, d_{\theta,\phi})$.
\end{theorem}

\begin{proof}
  Let $f\in\mathcal{M}$, set
  \[
    r_0:=1, \qquad  \alpha := \frac{\phi^{-1}\left(\tfrac{1}{n}\right)}{8nC_{\phi}},
  \]
  let $r \in (0,r_0)$ and pick a $\gamma\in(\tfrac{r}{2C_{\phi}},\tfrac{r}{C_{\phi}})$. We define
  \[
    f_\gamma(x) :=  (1-\gamma)f(x) \oplus \gamma f(\theta).
  \]
  By Proposition~\ref{prop:StrContrDense} $f_\gamma$ is $(1-\gamma)$-Lipschitz and satisfies $d_{\theta,\phi}(f,f_\gamma)\leq C_{\phi} \gamma < r$.

  We now show that every $g\in\mathcal{M}$ with $d_{\theta,\phi} (g,f_\gamma) \leq \alpha r$ is an element of $\mathcal{G}_n$. Since $r<1$, we may use the triangle inequality together with Lemma~\ref{lem:localAndGlobalDistance} to obtain for $x,y\in \bar{B}(\theta, n)$
  \begin{align*}
    \rho(g(x),g(y)) &\leq \rho(g(x), f_\gamma(x)) + \rho(f_\gamma(x), f_\gamma(y)) + \rho( f_\gamma(y), g(y))\\
                  &\leq \frac{2\alpha r}{\phi^{-1}\left(\tfrac{1}{n}\right)} + (1-\gamma)\rho(x,y) + \frac{2\alpha r}{\phi^{-1}\left(\tfrac{1}{n}\right)} \leq \frac{r}{2nC_\phi} + (1-\gamma)\rho(x,y),
  \end{align*}
  where in the last step we used the definition of $\alpha$.

  Hence we conclude that for $x,y\in \bar{B}(\theta,n)$ with $\rho(x,y)\geq \frac{1}{n}$, we have
  \begin{align*}
      \frac{\rho(g(x), g(y))}{\rho(x,y)} &\leq \frac{r}{2nC_\phi} \frac{1}{\rho(x,y)} + (1-\gamma) \leq \frac{r}{2C_\phi}  + 1-\gamma < 1,
  \end{align*}
  where the last inequality follows from the definition of $\gamma$.
\end{proof}

\begin{theorem}\label{thm:RakotchOnBoundedSubsetsIsSigmaPorous2}
  For each $n\in\mathbb{N}$ the set $\mathcal{M}\setminus\mathcal{G}_n$ is porous in  $(\mathcal{M}, d_{\theta,s})$.
\end{theorem}

\begin{proof}
  Let $f\in\mathcal{M}$, set
  \[
    r_0:=1,\qquad \alpha:= \frac{1}{4n(1+n^s)},
  \]
  let $r \in (0,r_0)$ and pick a $\gamma\in(\frac{r}{2},r)$. We define
  \[
    f_\gamma(x) := (1-\gamma)f(x)  \oplus \gamma f(\theta).
  \]
  By Proposition~\ref{prop:StrictContrDenseFords} $f_\gamma$ is $(1-\gamma)$-Lipschitz and satisfies $d_{\theta,s}(f,f_\gamma)\leq \gamma < r$.

  We now show that every $g\in\mathcal{M}$ with $d_{\theta,s} (g,f_\gamma) \leq \alpha r$ is an element of $\mathcal{G}_n$. Since $r<1$, we may use the triangle inequality together with Lemma~\ref{lem:sLocalToGlobal} to obtain for $x,y\in \bar{B}(\theta, n)$ that
  \begin{align*}
    \rho(g(x),g(y)) &\leq \rho(g(x), f_\gamma(x)) + \rho(f_\gamma(x), f_\gamma(y)) + \rho( f_\gamma(y), g(y))\\ & \leq \alpha r (1+\rho(x,\theta)^s) + (1-\gamma)\rho(x,y)  + \alpha r (1+\rho(y,\theta)^s)  \leq \frac{r}{2n} + (1-\gamma)\rho(x,y).
  \end{align*}
  Hence we conclude that for $x,y\in \bar{B}(\theta,n)$ with $\rho(x,y)\geq \frac{1}{n}$, we have
  \begin{align*}
      \frac{\rho(g(x), g(y))}{\rho(x,y)} &\leq \frac{r}{2n} \frac{1}{\rho(x,y)} + (1-\gamma) \leq \frac{r}{2}  + 1-\gamma < 1,
  \end{align*}
  where the last inequality follows from the definition of $\gamma$.
\end{proof}

Given a mapping $f\in \bigcap_{n\in\mathbb{N}}\mathcal{G}_n$, we set
\[
  c_{f,n} := \sup \left\{\frac{\rho(f(x),f(y))}{\rho(x,y)}\colon x,y\in \bar{B}(\theta,n) \;\text{and}\; \rho(x,y)\geq \frac{1}{n}\right\}
\]
and observe that by construction and by the definition of the sets~$\mathcal{G}_n$ the sequence $(c_{f,n})_{n\in\mathbb{N}}$ is an increasing sequence in $[0,1)$. For $M>0$ we define the function
\[
  \varphi_{f,M} \colon [0,\infty) \to [0,1], \qquad t \mapsto \sum_{n=M+1}^{\infty} c_{f,n} \chi_{\left[\frac{1}{n},\frac{1}{n-1}\right)}(t) + c_{M} \chi_{\left[\frac{1}{M},\infty\right)}(t)
\]
and note that it is well defined since for every $t$ only one summand is nonzero. Moreover note that it is a decreasing function satisfying $\varphi_{f,M}(t) < 1$ for $t>0$ and $\rho(f(x), f(y)) \leq \varphi_{f,M}(\rho(x,y)) \rho(x,y)$ for all $x,y\in \bar{B}(\theta,M)$. This shows that $f$ satisfies property~(R).

\begin{corollary}
  The typical mapping in $(\mathcal{M}, d_{\theta,\phi})$ and $(\mathcal{M}, d_{\theta,s})$ is Rakotch contractive on bounded sets, that is, the complement of the set of mappings $f$ with property~(\ref{it:RB}) is $\sigma$-porous.
\end{corollary}

\begin{proof}
  With $(\mathcal{G}_n)_{n\in\mathbb{N}}$ as in Theorem~\ref{thm:RakotchOnBoundedSubsetsIsSigmaPorous}, the given set is equal to
  \[
    \mathcal{M}\setminus\bigcap_{n\in\mathbb{N}}\mathcal{G}_n = \bigcup_{n\in\mathbb{N}}\mathcal{M}\setminus \mathcal{G}_n,
  \]
  which is a countable union of porous sets both for~$d_{\theta,\phi}$ and~$d_{\theta,s}$ and therefore $\sigma$-porous.
\end{proof}

As a direct consequence of these two results we obtain the following result regarding the existence of fixed points.

\begin{corollary}
  The typical mapping in $(\mathcal{M}, d_{\theta,\phi})$ and $(\mathcal{M}, d_{\theta,s})$ has a unique fixed point which can be reached by sequences of iterates. More precisely, the set of mappings without this property is $\sigma$-$\sqrt{\phi}$-porous in $(\mathcal{M}, d_{\theta,\phi})$ and $\psi_s$-porous in $(\mathcal{M}, d_{\theta,s})$. If $\phi$ satisfies (C5) it is $\sigma$-$\phi$-porous with respect to~$d_{\theta,\phi}$.
\end{corollary}

\begin{proof}
  We set
  \[
    \tilde{\mathcal{F}} := \bigcap_{n=1}^{\infty} \mathcal{G}_n \cap \mathcal{F}
  \]
  where $\mathcal{F}$ is the set from either Theorem~\ref{thm:BallIntoBall} or Theorem~\ref{thm:BallIntoBall2}. Observe that (C2) together with~\eqref{eq:PhiInvSmT} imply that $\sqrt{\phi(t)}\geq\phi(t)\geq t$ for $t$ small enough. Hence porosity implies $\phi$-porosity which in turn implies $\sqrt{\phi}$-porosity. Similarly, the inequality $\psi_s(t)\geq t$ for $t\leq 1$ implies that porosity implies $\psi_s$-porosity. Note that by the results of this section, the complement of this set satisfies the claimed porosity properties. Now let $f\in \tilde{\mathcal{F}}$ be given. Since $f\in\mathcal{F}$ and by taking into account Remark~\ref{rem:Balls}, there is an $M_f>0$ such that $f(\bar{B}(\theta,M))\subset \bar{B}(\theta,M)$ for all $M\geq M_f$, that is, $f$ is a self-mapping of $\bar{B}(\theta,M)$ for all large enough $M$. On the other hand, since $f\in \bigcap_{n=1}^{\infty} \mathcal{G}_n$ for every $M>0$, there is a decreasing function $\varphi_{f,M}\colon (0,\infty)\to (0,1)$ with
  \[
    \rho(f(x),f(y)) \leq \varphi_{f,M}(\rho(x,y)) \rho(x,y)
  \]
  for all $x,y\in \bar{B}(\theta,M)$ with $x\neq y$. This shows that for large enough balls around~$\theta$ the assumptions of Rakotch's fixed point theorem are satisfied. Hence for all $x\in X$ the sequence
  \[
    x_0 := x,\qquad x_{n+1}:=f(x_n)\qquad\text{for}\;n\in\mathbb{N}
  \]
  converges to a fixed point of $f$. Moreover, note that the above shows that $\rho(f(x),f(y))<\rho(x,y)$ for all $x\neq y$, which implies that the fixed point has to be unique.
\end{proof}

In the rest of this section, we want to show that the set of Rakotch contractions is a small subset of~$\mathcal{M}$ for both metrics $d_{\theta,\phi}$ and $d_{\theta,s}$. To this end, we need some technical preparations. For a mapping $f\colon X\to X$, we denote by $\omega_f$ the modulus of continuity of $f$, that is,
\[
  \omega_f(t) := \sup\{\rho(f(x),f(y))\colon x,y\in X, \rho(x,y)\leq t\}
\]
for $t \geq 0$. For $\mu\in (0,1)$ and $t_0>0$, we set
\[
  \mathcal{N}_{\mu,t_0} := \{f\in\mathcal{M}\colon \omega_f(t_0) \leq \mu t_0\}.
\]

\begin{theorem}\label{thm:modcont}
  For every $t_0>0$ and every $\mu\in(0,1)$, the set $\mathcal{N}_{\mu,t_0}$ is $\phi$-porous in  $(\mathcal{M}, d_{\theta,\phi})$.
\end{theorem}

We also have a similar result for the metric $d_{\theta,s}$.

\begin{theorem}\label{thm:modcont2}
  For every $t_0>0$ and every $\mu\in(0,1)$, the set $\mathcal{N}_{\mu,t_0}$ is $\psi_s$-porous in  $(\mathcal{M}, d_{\theta,s})$.
\end{theorem}

The proofs of these theorems are given in Section~\ref{sec:RakotchSmall}. Using these results we are now able to show the following corollary which generalises Corollary~4.3 of~\cite{Strobin} which states that the set of Rakotch contractions on an unbounded closed and convex subset of a Hilbert space is meagre for the topology of uniform convergence on bounded sets. Since the main tool in~\cite{Strobin} is a generalisation of the Kirszbraun-Valentine extension theorem which is not available outside Hilbert spaces, our proofs, given in Section~\ref{sec:RakotchSmall}, use completely different methods.

\begin{corollary}
  The set $\{f\in\mathcal{M}\colon \omega_f(t) < t\;\text{for some}\;t>0\}$ is $\sigma$-$\phi$-porous in $(\mathcal{M}, d_{\theta,\phi})$ and $\sigma$-$\psi_s$-porous in $(\mathcal{M}, d_{\theta,s})$. In particular the set of Rakotch contractions is $\sigma$-$\phi$-porous and $\sigma$-$\psi_s$-porous in $(\mathcal{M}, d_{\theta,\phi})$ and $(\mathcal{M}, d_{\theta,s})$, respectively.
\end{corollary}

The following proof follows the one of Corollary~4.2 in~\cite{Strobin}.

\begin{proof}
  We denote by $\mathbb{Q}_{+}$ the set of positive rationals and observe that the set
  \[
    \bigcup_{q\in\mathbb{Q}_{+}} \bigcup_{n\in\mathbb{N}} \mathcal{N}_{\frac{n-1}{n},q}
  \]
  is $\sigma$-$\phi$-porous and $\sigma$-$\psi_s$-porous in $(\mathcal{M}, d_{\theta,\phi})$ and $(\mathcal{M}, d_{\theta,s})$, respectively. For a mapping $f$ in the complement of the above set we have $\omega_{f}(q) = q$ for every positive rational number~$q$. Given an irrational number $t\in [0,\infty)$ we pick sequences $(q_n)_{n=1}^{\infty}$ and $(r_n)_{n=1}^{\infty}$ of rational numbers converging to $t$ from above and below, respectively. Then we have $r_n = \omega_f(r_n) \leq \omega_f(t) \leq \omega_f(q_n)=q_n$ for all $n$ which by taking $n\to\infty$ implies that $\omega_f(t)=t$.
\end{proof}

\section{Proof of Theorems~\ref{thm:BallIntoBall} and~\ref{thm:BallIntoBall2}}
In Theorems~\ref{thm:BallIntoBall} and~\ref{thm:BallIntoBall2} we claim three different but closely related porosity results. Since our proofs for all of them use the same construction and only differ in the choice of certain parameters, we present them here in parallel. For this aim denote
\[
  \mathcal{F} = \{f\in\mathcal{M} \colon f \textnormal{ satisfies (P1)}\}.
\]
In order to show that the complement of this set satisfies the claimed porosity conditions at each of its points, let
\[
  f \in \mathcal{M}\setminus\mathcal{F}.
\]
Now we have to choose the required parameters which are different for each of the cases.

We set $r_0:=\min\{1,\phi(\eta_\phi)\}$ in the case of $d_{\theta,\phi}$ and $r_0:=1$ in the case of $d_{\theta,s}$. Moreover, we set
\[
  \alpha := \frac{1}{\sqrt{12C_\phi(\rho(\theta,f(\theta))+2)}}, \qquad\tilde{\alpha} := \frac{1}{12C_\phi(\rho(\theta,f(\theta))+2)+1}
\]
and
\[
  \alpha_s := \frac{1}{6(2+\rho(f(\theta),\theta))}.
\]
Let us briefly comment on the role of these constants. The constant $\alpha$ will be used to show that $\mathcal{M}\setminus \mathcal{F}$ is $\sqrt{\phi}$-porous for the metric $d_{\theta,\phi}$ in the general case and $\tilde{\alpha}$ will be employed to show $\phi$-porosity under the additional assumption that $\phi$ satisfies condition~(C5). Finally, the constant $\alpha_s$ will be used to show that $\mathcal{M}\setminus \mathcal{F}$ is $\psi_s$-porous for the metric $d_{\theta,s}$. Since hyperbolic spaces, being geodesic, cannot contain isolated points and $\phi$, $\sqrt{\phi}$ and~$\psi_s$ satisfy the assumptions of Lemma~\ref{lem:PorousMichael}, we use the characterisation given there to prove the claimed porosity properties. We continue now with the common construction for all proofs but distinguish two cases for the choice of $\gamma$.

For $r \in (0,r_0)$ we set $\gamma := \frac{r}{3}$ in the case of $d_{\theta,s}$ and $\gamma:=\frac{r}{3C_\phi}$ in the case of $d_{\theta,\phi}$. We define
\[
  f_\gamma(x):= (1-\gamma)f(x) \oplus \gamma f(\theta),\qquad x\in X
\]
and observe that
\[
  f_\gamma(\theta) =  (1-\gamma)f(\theta) \oplus \gamma f(\theta) = f(\theta).
\]
Next, we choose
\[
  M_f:= \frac{1 + \rho(f(\theta),\theta)}{\gamma} >1.
\]

Observe that by Propositions~\ref{prop:StrContrDense} and~\ref{prop:StrictContrDenseFords} we have $f_\gamma \in \mathcal{M}$, $d_{\theta,\phi}(f,f_\gamma)<r$ and $d_{\theta,s}(f,f_\gamma)<r$. Having constructed the centre of the required ball, the following three lemmas now show that in all three cases we indeed have such a ball inside $\mathcal{F}$.

\begin{lemma}\label{lem:CInF}%
  For $g \in B_{d_{\theta,\phi}}(f_\gamma, \phi^{-1}((\alpha r)^2))$, we have $g \in\mathcal{F}$.
\end{lemma}

\begin{proof}
  In order to check whether $g$ satisfies property (P1), let $z\in \bar{B}(\theta, M_f)$
  be given. In order to show that $g(z)\in \bar{B}(\theta,M_f)$, we first need to observe that
  \begin{align*}
    \phi^{-1}((\alpha r)^2) &= \phi^{-1}\left(\frac{r^2}{12C_\phi(\rho(\theta,f(\theta))+2)}\right) = \phi^{-1}\left(r\frac{1}{6C_\phi\frac{\rho(\theta,f(\theta))+2}{r}+6C_\phi\frac{\rho(\theta,f(\theta))+2}{r}}\right)\\
                            & \leq \phi^{-1}\left(r\frac{1}{2(M_f+2)}\right) \leq r \phi^{-1}\left( \frac{1}{2(\ceil{M_f}+1)}\right),
  \end{align*}
  where the first inequality follows from $2M_f+4 \leq 6C_\phi\frac{\rho(\theta,f(\theta))+2}{r}+ 6C_\phi\frac{\rho(\theta,f(\theta))+2}{r}$ and the last inequality holds by~\eqref{eq:PhiInvConvIneq} of Remark~\ref{rem:PhiInvConv} since $\frac{1}{2(\ceil{M_f}+1)}< \frac{1}{M_f} = \frac{r}{3C_\phi (2+ \rho(\theta,f(\theta)))} < r_0 \leq \phi(\eta_\phi)$ by the choice of $r_0$ and~$\gamma$.
  Since $\rho(z,\theta) \leq \ceil{M_f}+1$ and $\ceil{M_f}+1 > \frac{1}{\phi(\eta_\phi)}$ we may use Lemma~\ref{lem:localAndGlobalDistance} to find that
  \begin{align*}
      \rho(g(z),\theta) &\leq \rho(g(z), f_\gamma(z))  + \rho(f_\gamma(z), f_\gamma(\theta)) + \rho(f_\gamma(\theta),\theta)\leq r + (1-\gamma) M_f + \rho(f_\gamma(\theta),\theta)\\ & \leq 1 +\rho(f(\theta),\theta) + (1-\gamma) M_f = \gamma M_f + (1-\gamma) M_f = M_f
  \end{align*}
  where in the last inequality we use that $r<1$. Hence, we have shown that $g\in\mathcal{F}$, as asserted.
\end{proof}

\begin{lemma}\label{lem:CInF2}%
  Let $\phi$ satisfy~\eqref{eq:C5} and let $g \in B_{d_{\theta,\phi}}(f_\gamma, \phi^{-1}(\tilde{\alpha} r))$. Then, $g \in\mathcal{F}$.
\end{lemma}

\begin{proof}
  Using (C5) we observe that
  \begin{align*}
    \phi^{-1}(\tilde{\alpha} r) & =  \phi^{-1}\left(\frac{r}{12C_\phi(\rho(\theta,f(\theta))+2)+1}\right) \\ &\leq \phi^{-1}(r) \phi^{-1}\left(\frac{1}{6C_\phi\frac{\rho(\theta,f(\theta))+2}{r}+6C_\phi\frac{\rho(\theta,f(\theta))+2}{r}}\right)\\
                            & \leq r\phi^{-1}\left(\frac{1}{2(M_f+2)}\right) \leq r \phi^{-1}\left(\frac{1}{2(\ceil{M_f}+1)}\right)
  \end{align*}
  where we used $r<\phi(\eta_\phi)$ and~\eqref{eq:PhiInvSmT} of Remark~\ref{rem:PhiInvConv} and proceed as in the proof of Lemma~\ref{lem:CInF}.
\end{proof}

\begin{lemma}
  For $g \in B_{d_{\theta,s}}(f_\gamma, \psi_s^{-1}(\alpha_s r))$, we have $g \in\mathcal{F}$.
\end{lemma}

\begin{proof}
  In order to check whether $g$ satisfies property (P1), let
  \[
    z \in \bar{B}(\theta,M_f)
  \]
  be given. We show that $g(z)\in \bar{B}(\theta,M_f)$. To this end, we first observe that
  \begin{align*}
    (\alpha_s r)^s &= \frac{r^s}{6^s(2+\rho(f(\theta),\theta))^s} = \frac{\gamma^s}{2^s(2+\rho(f(\theta), \theta))^s} = \frac{1}{2^sM_f^s} \leq \frac{1}{1+M_f^s}
  \end{align*}
  and then use the triangle inequality, Lemma~\ref{lem:sLocalToGlobal} and $\rho(z,\theta)\leq M_f$ to obtain
  \begin{align*}
    \rho(g(z),\theta) &\leq \rho(g(z), f_\gamma(z))  + \rho(f_\gamma(z), f_\gamma(\theta)) + \rho(f_\gamma(\theta),\theta)\\ & \leq (1 + \rho(z,\theta)^s) (\alpha_sr)^s + (1-\gamma) M_f + \rho(f(\theta),\theta)\\ & \leq \frac{1+\rho(z,\theta)^s}{1+M_f^s} + (1-\gamma) M_f + \rho(f(\theta),\theta)\\
                   & \leq 1 +\rho(f(\theta),\theta) + (1-\gamma) M_f = \gamma M_f + (1-\gamma) M_f = M_f,
  \end{align*}
  as required.
\end{proof}

\section{Proof of Theorems~\ref{thm:modcont} and~\ref{thm:modcont2}}\label{sec:RakotchSmall}

In order to prove that the set of Rakotch contractions is $\sigma$-$\phi$-porous in $(\mathcal{M},d_{\theta,\phi})$ and $\sigma$-$\psi_{s}$-porous in $(\mathcal{M},d_{\theta,s})$, we need a number of lemmas on the perturbation of nonexpansive mappings.

\begin{lemma} \label{lem:minModul}
  For each $t_0 > 0$, $\lambda \in (0,1)$ and $x_0, y_0, v \in X$ with $\rho(x_0,y_0) = t_0$, there exists a mapping $\tau \in \mathcal{M}$ with the following properties:
  \begin{enumerate}[(i)]
  \item $\rho(\tau(x_0),\tau(y_0))>\lambda t_0$,
  \item $\tau(x)=v$ for $\rho(x,x_0)\geq t_0$ and
  \item $\rho(\tau(x),v) \leq t_0$ for all $x\in X$.
  \end{enumerate}
\end{lemma}

\begin{proof}
  Let $t_0>0$, $\lambda \in (0,1)$ and $x_0, y_0, v \in X$ with $\rho(x_0,y_0) = t_0$ be given. Choose some $0<\varepsilon\leq(1-\lambda)t_0$. Since $X$ is unbounded, we can find a point $u\in X$ with $\rho(u,v)=\lambda t_0 + \varepsilon$. We now define
  \[
    \gamma(x) := \max\left\{\frac{t_0-\rho(x,x_0)}{t_0}, 0 \right\} \qquad\text{and}\qquad
    \tau(x) = (1-\gamma(x)) v \oplus \gamma(x) u.
  \]
  Note that using Lemma~\ref{lem:LipschitzSubset} we obtain $\Lip \gamma \leq \frac{1}{t_0}$. By definition~$\tau$ satisfies
  \[
    \rho(\tau(x_0), \tau(y_0)) = \rho(u,v) = \lambda t_0 + \varepsilon > \lambda t_0 \qquad\text{and}\qquad \tau(x)=v \text{ for } \rho(x,x_0) \geq t_0.
  \]
  It remains to show that $\tau \in \mathcal{M}$. By Lemma~\ref{lem:LipschitzSubset} it is enough to check the Lipschitz condition for $x, y \in B(x_0,t_0)$ since outside of this ball the mapping is constant.
  Indeed, using~\eqref{eq:HypLinComb2} together with $\Lip\gamma \leq \frac{1}{t_0}$ and $\rho(u,v)=\lambda t_0 + \varepsilon$, we have
  \begin{align*}
    \rho(\tau(x),\tau(y)) &= \rho((1-\gamma(x)) v \oplus \gamma(x) u, (1-\gamma(y)) v \oplus \gamma(y) u)\\
                          &= |\gamma(x)-\gamma(y)| \, \rho(u,v) \leq \frac{\rho(x, y)}{t_0}(\lambda t_0+\varepsilon)  \leq  \rho(x, y),
  \end{align*}
  which finishes the proof.
\end{proof}

\begin{lemma}\label{lem:slowDecreaseToZero}
  For every $z\in X$, $R>0$ and $\varepsilon>0$, the function
  \[
    \lambda_{z,R,\varepsilon}: X\rightarrow [0,1],  \qquad x \mapsto
    \begin{cases}
      1 & \text{for}\; \rho(x,z)\leq R\\
      1 - \varepsilon(\rho(x,z)-R) & \text{for}\; R< \rho(x,z)< R + \frac{1}{\varepsilon}\\
      0 & \text{for}\; \rho(x,z)\geq R + \frac{1}{\varepsilon}
    \end{cases}
  \]
  equals one on $\bar{B}(z,R)$, satisfies $\Lip \lambda_{z,R,\varepsilon} \leq \varepsilon$ and vanishes outside the ball $B(z, R+\frac{1}{\varepsilon})$.
\end{lemma}

\begin{proof}
  We only have to show that $\Lip \lambda_{z,R,\varepsilon} \leq \varepsilon$. In view of Lemma~\ref{lem:LipschitzSubset}, it suffices to check points $x,y\in X$ with $R\leq \rho(x,z),\rho(y,z) \leq R+\frac{1}{\varepsilon}$. For these, we have
  \[
    |\lambda_{z,R,\varepsilon}(x)-\lambda_{z,R,\varepsilon}(y)| = |\varepsilon \rho(x,z) - \varepsilon\rho(y,z)| \leq \varepsilon \rho(x,y),
  \]
  as required.
\end{proof}

A lemma similar to the following one has recently been obtained by M.~Dymond in~\cite{Dymond2021}.

\begin{lemma}\label{lem:getConstBall}
  For every $z\in X$, $R>0$ and $\varepsilon >0$, the mapping
  \[
    \pi_{z,R,\varepsilon} \colon X \to X, \qquad x \mapsto
    \begin{cases}
      z & \text{for}\;\rho(x,z) < R\\
      \left(1- \frac{R\lambda(x)}{\rho(x,z)} \right) x \oplus \frac{R \lambda(x)}{\rho(x,z)}z & \text{for}\; \rho(x,z) \geq R
    \end{cases},
  \]
  where $\lambda=\lambda_{z,R,\frac{\varepsilon}{R}}$ is the mapping from Lemma~\ref{lem:slowDecreaseToZero}, is a $(1+\varepsilon)$-Lipschitz mapping so that $\pi_z(x) = z$ for $x\in B(z,R)$, $\pi_z(x)=x$ for $x\in X\setminus B(z, R(1+\frac{1}{\varepsilon}))$ and  $\rho(\pi_z(x),x) \leq R$ for all $x\in X$.
\end{lemma}

\begin{proof}
  In view of Lemma~\ref{lem:LipschitzSubset}, we only have to check the Lipschitz condition outside the ball $B(z,R)$  since on $B(z,R)$ the mapping $\pi_{z,R,\varepsilon}$ is constant.
  For $\rho(x,z),\rho(y,z) \geq R$ we use the triangle inequality, \eqref{eq:HypIneq}, \eqref{eq:HypLinComb2} and the fact that by Lemma~\ref{lem:slowDecreaseToZero} $\Lip \lambda \leq \frac{\varepsilon}{R}$ to obtain that
  \begin{align*}
    \rho(\pi_{z,R,\varepsilon}(x),\pi_{z,R,\varepsilon}(y)) & = \rho\left(\left(1-\frac{R\lambda(x)}{\rho(x,z)}\right)x\oplus \frac{R \lambda(x)}{\rho(x,z)}z, \left(1-\frac{R\lambda(y)}{\rho(y,z)}\right)y\oplus \frac{R \lambda(y)}{\rho(y,z)}z\right) \\
                        &\leq \rho\left(\left(1-\frac{R\lambda(x)}{\rho(x,z)}\right)x\oplus \frac{R \lambda(x)}{\rho(x,z)}z, \left(1-\frac{R\lambda(x)}{\rho(x,z)}\right)y\oplus \frac{R \lambda(x)}{\rho(x,z)}z\right) \\
                        & \quad + \rho\left(\left(1-\frac{R\lambda(x)}{\rho(x,z)}\right)y\oplus \frac{R \lambda(x)}{\rho(x,z)}z, \left(1-\frac{R\lambda(y)}{\rho(y,z)}\right)y\oplus \frac{R \lambda(y)}{\rho(y,z)}z\right) \\
                        & \leq \left(1-\frac{R\lambda(x)}{\rho(x,z)}\right) \rho(x,y) + \left|\frac{R \lambda(x)}{\rho(x,z)}- \frac{R \lambda(y)}{\rho(y,z)}\right| \rho(y,z) \\
                        & \leq \left(1-\frac{R\lambda(x)}{\rho(x,z)}\right) \rho(x,y) +
                        \left|\frac{R \lambda(x)}{\rho(x,z)}- \frac{R \lambda(x)}{\rho(y,z)}\right| \rho(y,z)\\ & \quad + \left|\frac{R\lambda(x)}{\rho(y,z)}- \frac{R\lambda(y)}{\rho(y,z)}\right|\rho(y,z)\\
                        & \leq \left(1-\frac{R\lambda(x)}{\rho(x,z)}\right) \rho(x,y) + \frac{R\lambda(x)}{\rho(x,z)} \rho(x,y)+ R\Lip \lambda \rho(x,y) \\ &\leq (1+\varepsilon) \rho(x,y).
  \end{align*}
  Note that by Lemma~\ref{lem:slowDecreaseToZero}, for $x\in X$ with $\rho(x,z)\geq  R(1+\frac{1}{\varepsilon})$ we have $\lambda_{z,R,\frac{\varepsilon}{R}}(x)=0$ and hence $\pi_z(x)=x$. Finally observe that for $x\in X\setminus B(z,R)$ we have
  \begin{align*}
    \rho(x,\pi_z(x)) = \rho\left(x, \left(1- \frac{R\lambda(x)}{\rho(x,z)} \right) x \oplus \frac{R \lambda(x)}{\rho(x,z)}z\right) = R\lambda(x) \leq R
  \end{align*}
  by~\eqref{eq:HypLinComb} and Lemma~\ref{lem:slowDecreaseToZero}. This completes the proof.
\end{proof}

\begin{remark}
  In the case where $X$ is an unbounded, closed and convex subset of a Banach space, we can give a more elegant proof for a slight modification of the above lemma using the inequality
  \[
    \|P_Bx-P_By\| \leq \frac{4R}{\|x\|+\|y\|} \|x-y\|
  \]
  for the radial retraction onto the ball $B=\bar{B}(0,R)$
  \[
    P_B(x) = \begin{cases} x & \text{for}\; \|x\|\leq R\\ \frac{Rx}{\|x\|} & \text{otherwise}\end{cases}
  \]
  which is due to C.~F.~Dunkl and K.~S.~Williams, see~\cite{DW1964:SimpleInequality} or~\cite{dFK1967:RadialProjection}, together with the observation that the mapping $x\mapsto x-P_Bx$ is nonexpansive; see, for example, Lemma~4.2 in~\cite{Medjic}. Now we use the function $\lambda_{z,\frac{4R}{\varepsilon},\frac{\varepsilon}{2r}}$ from Lemma~\ref{lem:slowDecreaseToZero} which is one on the ball $\bar{B}(z,\frac{4R}{\varepsilon})$ and satisfies $\Lip \lambda < \frac{\varepsilon}{2R}$ and set
  \[
    \pi_z(x) = x - \lambda(x)P_B(x-z) = z + (1-\lambda(x))(x-z) + \lambda(x)((x-z)-P_B(x-z)).
  \]
  The above-mentioned results imply that $\pi_z$ is nonexpansive on the ball $\bar{B}(z,4R/\varepsilon)$. Outside this ball, we use the triangle inequality, that the image of $P_B$ is a ball of radius $R$ and the above bound on the Lipschitz constant of $\lambda$ to obtain
  \begin{align*}
    \|\pi_z(x)-\pi_z(y)\| & = \|x-y+\lambda(x)P_B(x-z)-\lambda(y)P_B(y-z)\|\\ & \leq \|x-y\| + \|\lambda(x)P_B(x-z)-\lambda(y)P_B(y-z)\|\\
                          & \leq \|x-y\| + \|\lambda(x)P_B(x-z)-\lambda(y)P_B(x-z)\|\\ &\quad  +  \|\lambda(y)P_B(x-z)-\lambda(y)P_B(y-z)\|\\
                          & \leq \|x-y\| + \Lip \lambda R \|x-y\| + \frac\varepsilon2 \|x-y\| \leq (1+\varepsilon) \|x-y\|.
  \end{align*}
  Since $\pi_z$ is continuous, we may use Lemma~\ref{lem:LipschitzSubset} to conclude that $\Lip\pi_z\leq 1+\varepsilon$. While the results on $P_B$ and $x\mapsto x-P_Bx$ generalise to hyperbolic space, the above computation uses the translation invariance of the metric induced by the norm and hence does not work for general hyperbolic spaces.
\end{remark}

The following lemma will be the crucial step of our proofs where we need to construct a nearby mapping with a large enough modulus of continuity.

\begin{lemma}\label{lem:LargeModCont}
  Let $z\in X$, $\gamma,\lambda\in(0,1)$, $t_0\in(0,\infty)$, $R > 0$ and $f\in\mathcal{M}$ be a strict contraction with $\Lip f \leq 1-\gamma$. There are $x_0,y_0\in \bar{B}(z,R+\frac{t_0}{\gamma})$ with $\rho(x_0,y_0)=t_0$ and a nonexpansive mapping $g\colon X\to X$ with $g(x)=f(x)$ for $x\in B(z,R)$, $\rho(g(x),f(x)) \leq 2t_0$ for $x\in X$ and $\rho(g(x_0),g(y_0))>\lambda t_0$.
\end{lemma}

\begin{proof}
  We set
  \[
    \delta:=\frac{\gamma}{1-\gamma} \qquad\text{and}\qquad S:= t_0\left(1+\frac{1}{\delta}\right) = \frac{t_0}{\gamma}.
  \]
  Since $X$ is unbounded, we can pick a point $x_0\in X$ with $\rho(x_0,z)=R+S$ and a point $y_0\in [z,x_0]$ with $\rho(x_0,y_0)=t_0$. Observe that by Lemma~\ref{lem:getConstBall}, we can find a $(1+\delta)$-Lipschitz mapping $\pi:=\pi_{x_0,t_0,\delta}$ on $X$ with $\pi|_{B(x_0,t_0)} = x_0$ and which is equal to the identity outside of $B(x_0,S)$. Let $\tau\in \mathcal{M}$ be the mapping from Lemma~\ref{lem:minModul} for $v=f(x_0)$ satisfying $\rho(\tau(x_0),\tau(y_0)) > \lambda t_0$.  We now define
  \begin{align*}
    g\colon X\to X, \qquad x\mapsto
    \begin{cases}
      \tau(x) & \text{for}\; \rho(x,x_0) \leq t_0\\
      f(\pi(x)) & \text{for}\; \rho(x,x_0)> t_0
    \end{cases}.
  \end{align*}
  Note that by construction $\tau(x)=f(x_0)=f(\pi(x))$ for all $x\in X$ with $\rho(x,x_0)=t_0$. Hence $g$ is continuous and we may use Lemma~\ref{lem:LipschitzSubset} to deduce that $g$ is a nonexpansive mapping because
  \[
    \rho(f(\pi(x)), f(\pi(y))) \leq \Lip f (1+\delta)\rho(x,y) \leq \rho(x,y)
  \]
  by the definition of~$\delta$.

  Next we show that $g(x)= f(x)$ for $x\in B(z,R)$ and $\rho(f(x),g(x)) \leq 2 t_0$ otherwise. Since $\pi(x)=x$ for $\rho(x,x_0)\geq S$ and $\rho(x_0,z)=R+S$, we have $g(x) = f(x)$ for all $x\in X$ with $\rho(x,z)\leq R$.
  If $\rho(x,x_0)\leq t_0$ then by Lemma~\ref{lem:minModul} we have $\rho(\tau(x),f(x_0)) \leq t_0$ and obtain
  \begin{align*}
    \rho(f(x),g(x)) & = \rho(f(x),\tau(x)) \leq \rho(f(x),f(x_0))+ \rho(f(x_0),\tau(x))\\ & \leq  \rho(x,x_0) + \rho(f(x_0), \tau(x)) \leq 2 t_0.
  \end{align*}
  For $\rho(x,x_0) > t_0$ on the other hand we have
  \begin{align*}
    \rho(f(x),g(x)) &= \rho(f(x), f(\pi(x))) \leq \rho(x,\pi(x)) \leq t_0
  \end{align*}
  by Lemma~\ref{lem:getConstBall}.

  We are left to show that $\rho(g(x_0),g(y_0))>\lambda t_0$. Since $\rho(x_0,y_0)=t_0$, we have
  \[
    \rho(g(x_0),g(y_0)) =\rho(\tau(x_0), \tau(y_0)) > \lambda t_0
  \]
  which finishes the proof.
\end{proof}

We conclude this section with the proof that for $t_0 > 0$ and $\mu\in (0,1)$, the set
\[
  \mathcal{N}_{t_0,\mu} := \left\{f \in \mathcal{M} \colon \omega_f(t_0)  \leq \mu t_0\right\}
\]
is a $\phi$-porous subset of $(\mathcal{M}, d_{\theta,\phi})$ and a $\psi_s$-porous subset of  $(\mathcal{M}, d_{\theta,s})$.

\begin{proof}[Proof of Theorem~\ref{thm:modcont}]
  Let $f\in \mathcal{N}_{t_0,\mu}$ be given. We set
  \[
    r_0 := \min\{1,\phi(\eta_\phi)\}, \qquad \varepsilon := \min\left\{\frac{t_0 (1-\mu)}{8}, \frac{1}{2}\right\} \qquad \text{and} \qquad \alpha := \frac{\varepsilon}{4+16 t_0 C_{\phi}}.
  \]
  Let $r\in (0,r_0)$. Using Proposition~\ref{prop:StrContrDense}, we pick an $\tilde{f}\in \mathcal{M}$ with $d_{\theta,\phi}(f,\tilde{f})<\frac{r}{2}$ and
  \[
    \Lip \tilde{f} \leq 1 - \frac{r}{4C_{\phi}}.
  \]
  We set
  \[
    N:=\ceil[\Big]{\tfrac{4C_{\phi}t_0}{r}} \leq \frac{4C_{\phi}t_0+1}{r}
  \]
  and use Lemma~\ref{lem:LargeModCont} for $\tilde{f}$ and the parameters $z=\theta$, $\gamma=\frac{r}{4C_\phi}$, $R=N$ and $\lambda=\frac{1+\mu}{2}$ to obtain points $x_0,y_0\in \bar{B}(\theta,2N )$ and a nonexpansive mapping $g\in\mathcal{M}$ with $g(x)=\tilde{f}(x)$ for all $x\in B(\theta,N)$, $\rho(g(x),\tilde{f}(x))\leq 2t_0$ otherwise and $\rho(g(x_0),g(y_0)) >\frac{1+\mu}{2}t_0$.

  In order to show that $d_{\theta,\phi}(f,g)<r$ we observe that
  \begin{align*}
    d_{\theta,\phi}(f,g) & \leq d_{\theta,\phi}(f,\tilde{f}) + \sum_{n=N}^{\infty} \phi^{-1}\left(\tfrac{1}{n}\right) d_{\theta,n}(\tilde{f},g) < \frac{r}{2} + 2 t_0 \sum_{n=N}^{\infty} \phi^{-1}\left(\tfrac{1}{n}\right) \\ &< \frac{r}{2} + \frac{2t_0}{N} \sum_{n=N}^{\infty} n \phi^{-1}\left(\tfrac{1}{n}\right)  \leq \frac{r}{2} + \frac{2t_0 C_{\phi}}{N} \leq r.
  \end{align*}

  We now want to use Lemma~\ref{lem:PorousMichael} and observe that
  \begin{align*}
    \phi^{-1}(\alpha r) &= \phi^{-1}\left(\varepsilon\frac{r}{4+16 t_0 C_{\phi}}\right) = \phi^{-1}\left(\frac{\varepsilon}{4\frac{1+4t_0C_{\phi}}{r}}\right) \leq \phi^{-1}\left(\frac{\varepsilon}{4N}\right) \leq \varepsilon \phi^{-1}\left(\frac{1}{4N}\right)
  \end{align*}
  by definition of~$\alpha$ and the monotonicity and convexity of $\phi^{-1}$ on $(0,\phi(\eta_\phi))$---see Remark~\ref{rem:PhiInvConv}. Given a mapping $h\in\mathcal{M}$ with $d_{\theta,\phi}(g,h)\leq \phi^{-1}(\alpha r)$ the computation above allows us to apply Lemma~\ref{lem:localAndGlobalDistance} to obtain
  \begin{align*}
    \rho(h(x_0),h(y_0)) &\geq \rho(g(x_0),g(y_0)) - \rho(g(x_0),h(x_0)) - \rho(g(y_0),h(y_0))\\ &\geq \mu t_0 +4\varepsilon - 2\varepsilon  > \mu t_0.
  \end{align*}
  This shows that $\omega_h(t_0)>\mu t_0$ and hence $h\not\in \mathcal{N}_{t_0,\mu}$. By Lemma~\ref{lem:PorousMichael}, this shows that $\mathcal{N}_{t_0,\mu}$ is $\phi$-porous.
\end{proof}

The main structure of the following proof is very similar to the previous one and the main difference is the choice of different values for the parameters.

\begin{proof}[{Proof of Theorem~\ref{thm:modcont2}}]
  Let $f\in \mathcal{N}_{t_0,\mu}$ be given. We set
  \[
    r_0 := 1, \qquad \varepsilon := \min\left\{\frac{t_0 (1-\mu)}{8}, \frac{1}{2}\right\} \qquad \text{and} \qquad \alpha := \frac{\varepsilon^{1/s}}{1+2t_0+(4t_0)^{1/s}}.
  \]
  Let $r\in (0,r_0)$. Using Proposition~\ref{prop:StrictContrDenseFords}, we pick an $\tilde{f}\in \mathcal{M}$ with $d_{\theta,s}(f,\tilde{f})<r/2$ and
  \[
    \Lip \tilde{f} \leq 1 - \frac{r}{2}.
  \]
  We now use Lemma~\ref{lem:LargeModCont} for $\tilde{f}$ and the parameters $z=\theta$, $\gamma=\frac{r}{2}$, $R=\frac{(4t_0)^{1/s}}{r}$ and $\lambda=\frac{1+\mu}{2}$ to obtain points $x_0,y_0\in \bar{B}\left(\theta,\frac{2t_0+(4t_0)^{1/s}}{r}\right)$ and a nonexpansive mapping $g\in\mathcal{M}$ with $g(x)=\tilde{f}(x)$ for all $x\in B(\theta,R)$, $\rho(g(x),\tilde{f}(x))\leq 2t_0$ otherwise and $\rho(g(x_0),g(y_0)) >\frac{1+\mu}{2}t_0$.

  In order to show that $d_{\theta,s}(f,g)<r$ we observe that
  \begin{align*}
    d_{\theta,s}(f,g) & \leq d_{\theta,s}(f,\tilde{f}) + \sup_{\rho(x,\theta)\geq R} \frac{\rho(\tilde{f}(x),g(x))}{1+\rho(x,\theta)^s} < \frac{r}{2} + \frac{2t_0}{1+\frac{4t_0}{r^s}} \leq \frac{r}{2} + r^s \frac{2t_0}{4t_0} \leq r
  \end{align*}
  since $r<1$. We observe that
  \begin{align*}
    (\alpha r)^s &= \frac{\varepsilon r^s}{(1+2t_0+(4t_0)^{1/s})^s} \leq \frac{\varepsilon}{1+\left(\frac{2t_0+(4t_0)^{1/s}}{r}\right)^s}
  \end{align*}
  by definition of~$\alpha$ and since $(1+2t_0+(4t_0)^{1/s})^s\geq 1+(2t_0+(4t_0)^{1/s})^s$ and $r<1$. Given a mapping $h\in\mathcal{M}$ with $d_{\theta,s}(g,h)\leq (\alpha r)^s$ the computation above allows us to apply Lemma~\ref{lem:sLocalToGlobal} to obtain
  \[
    \rho(h(x_0),h(y_0)) \geq \rho(g(x_0),g(y_0)) - 2\varepsilon \geq \mu t_0 + 2\varepsilon > \mu t_0.
  \]
  This shows that $\omega_h(t_0)>\mu t_0$ and hence $h\not\in \mathcal{N}_{t_0,\mu}$. By Lemma~\ref{lem:PorousMichael}, this shows that $\mathcal{N}_{t_0,\mu}$ is $\psi_s$-porous.
\end{proof}

\section{Some extensions}

\subsection{Behaviour of the local Lipschitz constant}

The aim of this section is to extend one of the results of~\cite{Dymond2021} to the setting of nonexpansive mappings on unbounded complete hyperbolic spaces. More precisely, we want to prove the following theorem.

\begin{theorem}\label{thm:MaxLipConst}
  There is a $\sigma$-porous subset $\mathcal{N}\subset\mathcal{M}$ both for $d_{\theta,\phi}$ and $d_{\theta,s}$ such that for every $f\in\mathcal{M}\setminus\mathcal{N}$, the set
  \[
    R(f) := \{x\in X \colon \Lip(f,x) = 1\}
  \]
  is a residual subset of $X$.
\end{theorem}

Above, $\Lip(f,x)$ denotes the following local variant of the Lipschitz constant. For $f\in\mathcal{M}$, $x\in X$ and $r\in(0,\infty)$, we set
\[
  \Lip(f,x,r) := \sup\left\{\frac{\rho(f(x),f(y))}{\rho(x,y)}\colon y\in X,\; 0 < \rho(x,y) < r\right\}
\]
and
\[
  \Lip(f,x) := \lim_{r\to 0^+} \Lip(f,x,r).
\]
The latter quantity is often called the \emph{local Lipschitz constant of $f$ at $x$}.

The proofs given in this section are based on small modifications of ideas from~\cite{Dymond2021}. The first step is to construct a mapping which is an isometry on a suitable set and close to a given nonexpansive mapping. As a preparation, we need the following lemma on the perturbation of a nonexpansive mapping. For $a>0$, we call a set $\Gamma\subset X$ \emph{$a$-separated} if $\rho(x,y)\geq a$ for all $x,y\in \Gamma$ with $x\neq y$.

\begin{lemma}\label{lem:Local1}
  Let $a\in (0,1)$ and let $\Gamma \subset X$ be $a$-separated, nonempty and non-singleton. Furthermore, let $\varepsilon\in(0,1)$ and $f\in \mathcal{M}$. There exists a mapping $g\in\mathcal{M}$ with
  \[
    \rho(f(x),g(x)) < \frac{3\varepsilon}{4} \max\{1,\rho(x,\theta)\} \qquad \text{for all}\quad x\in X
  \]
  and
  \[
    \rho(g(y), g(z)) = \rho(y,z)
  \]
  for all $z\in\Gamma$ and all $y\in \bar{B}(z, \frac{\varepsilon a}{32})$.
\end{lemma}

\begin{proof}
  We set
  \[
    r:= \frac{\varepsilon a}{16}, \qquad R:=\frac{a}{4}, \qquad \delta:=\frac{\varepsilon}{4-\varepsilon}
  \]
  and denote for $z\in \Gamma$ by $\pi_z$ the mapping $\pi_{z,r,\delta}$ from Lemma~\ref{lem:getConstBall}. Since $r(1+\frac{1}{\delta})=R$ we have $\pi_z(x)=x$ outside the ball $B(z,R)$. We define a mapping
  \[
    g_0\colon X\to X,\qquad x \mapsto
    \begin{cases}
      f(x)        &\text{if}\; x\in X\setminus \bigcup_{z\in\Gamma} B(z,R)\\
      f(\pi_z(x)) &\text{if}\; x \in B(z,R)\;\text{for}\; z\in\Gamma
    \end{cases}.
  \]
  Note that $g_0$ is continuous. In order to obtain a bound for the Lipschitz constant of $g_0$, we first observe that by Lemma~\ref{lem:getConstBall} we have $\Lip g_0|_{B(z,R)} \leq \Lip f\, \Lip \pi_z \leq 1 + \delta$ for all $z\in \Gamma$. Moreover since $R=\frac{a}{4}$ and $\Gamma$ is $a$-separated, we have $B(z,2R)\cap B(z',R) = \emptyset$ for $z,z'\in \Gamma$ with $z\neq z'$ and hence on $B(z,2R)\setminus B(z,R)$ the mapping $g_0$ agrees with the nonexpansive mapping $f$. Therefore $B(z,2R)$ is covered by two sets on which the Lipschitz constant of $g_0$ is bounded by $1+\delta$. Since balls in hyperbolic spaces are convex and hence geodesic, we may use Lemma~\ref{lem:LipschitzSubset} to conclude that the restriction of $g_0$ to $B(z,2R)$ is Lipschitz with Lipschitz constant at most $1+\delta$. Since on $X\setminus \bigcup_{z\in\Gamma} \bar{B}(z,R)$ the mapping $g_0$ agrees with the nonexpansive mapping~$f$, we have obtained a cover of $X$ consisting of open sets on which the Lipschitz constant of $g_0$ is bounded by $1+\delta$. Given two points $x,y\in X$, the metric segment $[x,y]$ is an isometric copy of a compact interval and hence compact itself. Hence it is covered by finitely many of these open sets and we can apply Lemma~\ref{lem:LipschitzSubset} to the restriction of $g_0$ to $[x,y]$ and conclude that $\rho(g_0(x),g_0(y))\leq (1+\delta)\rho(x,y)$. Since $x,y$ were arbitrary, this shows that $\Lip g_0 \leq 1 +\delta$. In addition note that for $x\in B(z,R)$ the inequality
  \[
    \rho(f(x),g_0(x)) \leq \rho(x,\pi_z(x)) \leq r
  \]
  and for $x\in B(z,r)$ the equality $g_0(x)=g_0(z)$ holds. We now set
  \[
    g_1\colon X \to X, \qquad x \mapsto \left(1-\frac{\varepsilon}{4}\right) g_0(x) \oplus \frac{\varepsilon}{4} g_0(\theta)
  \]
  and note that
  \begin{equation}\label{eq:distg0g1}
    \begin{aligned}
      \rho(g_0(x),g_1(x)) & = \rho\left(g_0(x), \left(1-\frac{\varepsilon}{4}\right) g_0(x) \oplus \frac{\varepsilon}{4} g_0(\theta)\right) = \frac{\varepsilon}{4} \rho(g_0(x),g_0(\theta))\\
      &\leq \rho(x,\theta) (1+\delta)\frac{\varepsilon}{4} \leq \rho(x,\theta) \frac{\varepsilon}{2}
    \end{aligned}
  \end{equation}
  by~\eqref{eq:HypLinComb} and since $\Lip g_0\leq 1+\delta\leq 2$. Since by construction $g_0$ is constant on the balls $B(z,r)$, $z\in\Gamma$, the same is true for $g_1$. More precisely, we have $g_1(x)=g_1(z)$ for all $x\in B(z,r)$ and $z\in\Gamma$. Since $X$ is unbounded, for every $z\in \Gamma$ we may pick a point $p_z\in X$ with $\rho(g_1(z),p_z)=\frac{a}{3}$. Using these points, we now define $ g \colon X \to X$ by
  \[
    g(x) =
    \begin{cases}
      g_1(x) &\text{if } x \in X \setminus \bigcup_{z\in \Gamma} B(z,r)\\[1mm]
      \left(1-3 \frac{r - \rho(z,x)}{a}\right) g_1(z) \oplus 3 \frac{r - \rho(z,x)}{a} p_z & \text{if}\; x\in  B(z, r)\setminus B(z,\frac{r}{2}) \;\text{for}\;z\in \Gamma\\[1.5mm]      \left(1-\frac{3 \rho(z,x)}{a}\right) g_1(z) \oplus \frac{3 \rho(z,x)}{a} p_z &\text{if}\; x\in  B(z, \frac{r}{2}) \;\text{for}\;z\in \Gamma
    \end{cases}
  \]
  which is well defined since $\frac{3\rho(z,x)}{a}\leq \frac{3\varepsilon}{16}<1$ for $\rho(z,x)\leq \frac{r}{2}$.
  The mapping $g$ can be thought of as a perturbation of $g_1$ such that on each of the balls where $g_1$ is constant, we add something which forces the Lipschitz constant to be one. Note that the mapping $g$ is continuous. We now show that it is, in fact, nonexpansive using an argument similar to the one above to show that $\Lip g_0\leq 1+\delta$. Given $x,y\not\in B(z,r)$ for $z\in \Gamma$, we have
  \begin{align*}
    \rho(g(x),g(y)) &= \rho(g_1(x),g_1(y))  \leq \left(1-\frac{\varepsilon}{4}\right) \rho(g_0(x),g_0(y)) \leq \frac{4-\varepsilon}{4} (1+\delta) \rho(x,y)\\ &= \frac{4-\varepsilon}{4}\frac{4}{4-\varepsilon} \rho(x,y) = \rho(x,y)
  \end{align*}
  by~\eqref{eq:HypIneq} and $\Lip g_0\leq 1+\delta=\frac{4}{4-\varepsilon}$. For $z\in \Gamma$ and $x,y\in B(z,\frac{r}{2})$ using~\eqref{eq:HypLinComb}, we obtain
  \[
    \rho(g(x),g(y)) = \frac{3}{a}\rho(g_1(z),p_z) |\rho(x,z)-\rho(y,z)| = |\rho(x,z)-\rho(y,z)|\leq \rho(x,y)
  \]
  and similarly for $z\in \Gamma$ and $x,y\in B(z,r) \setminus B(z,\frac{r}{2})$. Since $r < \frac{a}{4}$, the ball $B(z,2r)$ for $z\in \Gamma$ does not intersect $B(z',r)$ for $z'\in \Gamma$, $z'\neq z$. Hence, the above computations show that $B(z,2r)$ is covered by three sets on which the Lipschitz constant of $g$ is bounded by one. Therefore we may deduce from Lemma~\ref{lem:LipschitzSubset} that $g$ is is nonexpansive on the ball $B(z,2r)$ for every $z\in\Gamma$. The first of the above computations shows that $g$ is nonexpansive on the set $\Gamma\setminus \bigcup_{z\in\Gamma} \bar{B}(z,r)$. We now have an open cover of $X$ with sets on which $g$ is nonexpansive. As above, for $x,y\in X$ we can now use the compactness of $[x,y]$ together with  Lemma~\ref{lem:LipschitzSubset} to conclude that $\rho(g(x),g(y)) \leq \rho(x,y)$. Since $x$ and $y$ were arbitrary, this shows that $g$ is nonexpansive.

  In order to estimate the distance between $g_1$ and $g$, first recall that $g_1(x)=g_1(z)$ for all $x\in B(z,r)$. Note that for $x\in B(z,\frac{r}{2})$ we have
  \[
    \rho(g_1(x),g(x)) = \rho\left(g_1(z), \left(1-\frac{3 \rho(z,x)}{a}\right) g_1(z) \oplus \frac{3 \rho(z,x)}{a} p_z\right) = \frac{3 \rho(z,x)}{a} \rho(g_1(z),p_z) \leq \frac{r}{2}
  \]
  by~\eqref{eq:HypLinComb} and since $\rho(g_1(z),p_z)=\frac{a}{3}$. For $x \in B(z,r) \setminus B(z,\frac{r}{2})$ we use again~\eqref{eq:HypLinComb} and $\rho(g_1(z),p_z)=\frac{a}{3}$ to conclude that
  \begin{align*}
    \rho(g_1(x),g(x)) &= \rho\left(g_1(z), \left(1-3 \frac{r - \rho(z,x)}{a}\right) g_1(z) \oplus 3 \frac{r - \rho(z,x)}{a} p_z \right)\\ & = 3 \frac{r - \rho(z,x)}{a} \rho(g_1(z),p_z) \leq \frac{r}{2}.
  \end{align*}
  Combing these inequalities we have $\rho(g_1(x),g(x))\leq  \frac{r}{2}$ for all $x\in X$. Using the triangle inequality together with the bounds acquired above, we obtain that
  \begin{align*}
    \rho(f(x),g(x)) & \leq \rho(f(x),g_0(x)) + \rho(g_0(x),g_1(x)) + \rho(g_1(x), g(x)) \leq r + \frac{\varepsilon}{2} \rho(x,\theta) + r  \\
                    & \leq \left(\frac{\varepsilon}{8} +\frac{\varepsilon}{2}\right) \max\{1,\rho(x,\theta)\} < \frac{3\varepsilon}{4} \max\{1,\rho(x,\theta)\}
  \end{align*}
  since $r<\frac{\varepsilon}{16}$.  Finally, given $z\in\Gamma$ and $y\in \bar{B}(z, \frac{\varepsilon a}{32}) =  \bar{B}(z, \frac{r}{2})$, by~\eqref{eq:HypLinComb} and the fact that $g(z)=g_1(z)$ we have
  \[
    \rho(g(y), g(z)) = \frac{3}{a}\rho(g_1(z),p_z) \rho(y,z) = \rho(y,z)
  \]
  which finishes the proof.
\end{proof}

With this tool at hand, we are now able to show the following crucial lemmas.

\begin{lemma}
  Let $a\in(0,1)$, $\Gamma \subset X$ be $a$-separated and $\lambda\in(0,1)$. Furthermore let $n \in \mathbb{N}$ be such that the set $\Gamma\cap B(\theta,n)$ is nonempty and non-singleton. Then
  \[
    \mathcal{N}_{\lambda, \Gamma, a, n} :=\{ f\in\mathcal{M} \colon \inf_{x\in\Gamma\cap B(\theta, n)} \Lip(f,x,a) \leq \lambda\}.
  \]
  is a porous subset of $(\mathcal{M}, d_{\theta,\phi})$.
\end{lemma}

\begin{proof}
  Let $f\in\mathcal{M}$ and $a\in (0,1)$. We set
  \[
    r_0 := 1, \qquad \alpha := \frac{a(1-\lambda)}{2^{9}C_\phi}\phi^{-1}\left(\frac{1}{n}\right)
  \]
  and note that $\alpha\in(0,1)$. Let $r \in (0,r_0)$ and set $\varepsilon:=\frac{r}{C_\phi}$. Let $g$ be given by the previous lemma. By Lemma~\ref{lem:Local1} we have $\rho(f(x),g(x)) < \frac{3r}{4C_\phi}\max\{\rho(x,\theta),1\}$ and hence
  \[
    d_{\theta,\phi}(f,g) \leq \sum_{n=1}^{\infty} \phi^{-1}\left(\tfrac{1}{n}\right) \frac{3r n}{4C_\phi} < r.
  \]
  For every $z\in\Gamma\cap B(\theta, n)$ pick a point $y_z\in B(\theta,n)$ such that $\rho(z,y_z) = \frac{r a }{2^{6}C_{\phi}}$ and observe that in particular $y_z \in B\left(z, \frac{\varepsilon a}{32} \right)$.

  Now let $h\in\mathcal{M}$ with $d_{\theta,\phi}(h,g) < \alpha r$, since $\alpha < \frac{1}{2} \phi^{-1}\left(\frac{1}{n} \right)$, we obtain from the triangle inequality together with Lemma~\ref{lem:localAndGlobalDistance} that
  \begin{align*}
    \rho(h(y_z), h(z)) & \geq \rho(g(y_z), g(z)) - \rho(g(y_z),h(y_z)) - \rho(g(z),h(z))\\ & \geq \rho(g(y_z), g(z)) - \frac{4\alpha r}{\phi^{-1}\left(\frac{1}{n}\right)} = \rho(y_z,z) - (1-\lambda)\frac{ra2^{-7}}{C_{\phi}} \\ &= \rho(y_z,z)\left(1-\frac{(1-\lambda)}{2}\right) = \frac{1+\lambda}{2} \rho(y_z,z)
  \end{align*}
  and infer that
  \[
    \inf_{x\in\Gamma\cap B(\theta,n)} \Lip(h,x,a) \geq \frac{1+\lambda}{2}  >\lambda
  \]
  which implies that $h \not\in \mathcal{N}_{\lambda, \Gamma, a, n}$.
\end{proof}

\begin{lemma}
  Let $a\in(0,1)$, $\Gamma \subset X$ be $a$-separated and $\lambda\in(0,1)$.
  Furthermore let $n \in \mathbb{N}$ such that $\Gamma\cap B(\theta,n)$ is nonempty and non-singleton. Then
  \[
    \mathcal{N}_{\lambda, \Gamma, a, n} :=\{ f\in\mathcal{M} \colon \inf_{x\in\Gamma\cap B(\theta, n)} \Lip(f,x,a) \leq \lambda\}.
  \]
  is a porous subset of $(\mathcal{M}, d_{\theta,s})$.
\end{lemma}

\begin{proof}
  Let $f\in\mathcal{M}$ and $a\in (0,1)$. We set
  \[
    r_0:=1, \qquad \alpha := \frac{a(1-\lambda)2^{-8}}{1+n^s}
  \]
  and note that $\alpha\in(0,1)$. Given $r \in (0,r_0)$ we set $\varepsilon:=r$. Let $g$ be given by Lemma~\ref{lem:Local1}. This mapping satisfies $\rho(f(x),g(x)) < \frac{3}{4}r \max\{\rho(x,\theta),1\}$ and hence
  \[
    d_{\theta,s}(f,g) \leq \sup_{x\in X}  \frac{3r\max\{1,\rho(x,\theta)\}}{4(1+\rho(x,\theta)^s)} <  r.
  \]
  For every $z\in\Gamma\cap B(\theta, n)$ pick a point $y_z\in B(\theta,n)$ such that $\rho(z,y_z) = r a 2^{-6}$ and observe that in particular $y_z \in B\left(z, \frac{\varepsilon a}{32} \right)$.

  Now let $h\in\mathcal{M}$ with $d_{\theta,s}(h,g) < \alpha r$. Using the triangle inequality and Lemma~\ref{lem:sLocalToGlobal} we obtain that that
  \begin{align*}
    \rho(h(y_z), h(z)) & \geq \rho(g(y_z), g(z)) - 2\alpha r(1+n^s) = \rho(y_z,z) - ra2^{-7} (1-\lambda)\\ & = \rho(y_z,z)\left(1-\frac{(1-\lambda)}{2}\right) = \frac{1+\lambda}{2} \rho(y_z,z)
  \end{align*}
  and infer that
  \[
    \inf_{x\in\Gamma\cap B(\theta,n)} \Lip(h,x,a) \geq \frac{1+\lambda}{2}  >\lambda
  \]
  which implies that $h \not\in \mathcal{N}_{\lambda, \Gamma, a, n}$.
\end{proof}

The following lemma can be proved completely analogously to the corresponding one in~\cite{Dymond2021}. Since the proof is quite short, we include it nevertheless to keep this article self-contained.

\begin{lemma}
  Let $f\in\mathcal{M}$. Then the set
  \[
    R(f) := \{x\in X \colon \Lip(f,x) = 1\}
  \]
  is a $G_\delta$ subset of $X$.
\end{lemma}

\begin{proof}
  It is enough to recognise that the set $R(f)$ may be written as
  \[
    \bigcap_{\lambda \in \mathbb{Q}\cap(0,1)} \bigcap_{n\in \mathbb{N}} \left\{x\in X \colon \exists y\in B(x,\tfrac{1}{n}) : \rho(f(y), f(x)) > \lambda \rho(y,x)\right\}
  \]
  because every set in the intersection is open in $X$.
\end{proof}

\begin{proof}[Proof of Theorem \ref{thm:MaxLipConst}]
  For each $j\in\mathbb{N}$ define $\Gamma_j$ as a maximal $2^{-j}$-separated subset of~$X$. Such a maximal $2^{-j}$-separated subset exists by the Kuratowski–Zorn lemma since the $2^{-j}$-separated subsets are partially ordered by inclusion and for chains the union defines a maximal element. We set $a_{j,k} = 2^{-j-k}$ and note that for every $k\in\mathbb{N}$ every $\Gamma_j$ is also $a_{j,k}$-separated. We define the subset $\mathcal{N}\subset \mathcal{M}$ by
  \[
    \mathcal{N} : = \bigcup_{n\in\mathbb{N}} \;\bigcup_{\lambda \in \mathbb{Q} \cap (0,1)}\; \bigcup_{j\in\mathbb{N}} \;\bigcup_{k\in \mathbb{N}}  \mathcal{N}_{\lambda, \Gamma_j, a_{j,k}, n}
  \]
  where the sets $\mathcal{N}_{\lambda, \Gamma, a_{j,k}, n}$ are defined as in the previous lemmas. The lemmas of this section imply that $\mathcal{N}$ is a $\sigma$-porous set and we are left to show that for $f\in \mathcal{M}\setminus \mathcal{N}$ the set $R(f)$ is residual. Since $R(f)$ is a  $G_\delta$ subset of $X$ it suffices to show that it is dense. More precisely, we show that the dense subset $\bigcup_{j\in\mathbb{N}} \Gamma_j$ is contained in $R(f)$, which in turn implies that $R(f)$ is dense in $X$. For a given $f\in\mathcal{M}\setminus \mathcal{N}$ fix $j\in\mathbb{N}$ and let $x\in\Gamma_j$. We choose $n\in\mathbb{N}$ large enough so that $x\in B(\theta,n)$ and deduce from $f \not\in \bigcup_{\lambda \in \mathbb{Q} \cap (0,1)} \bigcup_{k\in \mathbb{N}}  \mathcal{N}_{\lambda, \Gamma_j, a_{k,j}, n}$ that
  \[
    \Lip(f,x,a_{j,k}) \geq \inf_{z\in\Gamma_j\cap B(\theta, n)} \Lip(f,z,a_{j,k}) > \lambda
  \]
  for every $\lambda\in (0,1)$ and every $k\in\mathbb{N}$, and hence
  \[
    \Lip(f,x) = 1,
  \]
  which shows that $x\in R(f)$.
\end{proof}

\subsection{On a metric of pointwise convergence}
Let $X$ be a complete \emph{separable} hyperbolic space. Recall that the limit of a pointwise convergent sequence of nonexpansive mappings is again a nonexpansive mapping; see e.g. Proposition~1.8 in~\cite[p.~6]{Weaver}. Since $X$ is separable, we may pick a dense sequence $(z_n)_{n=1}^\infty$ in~$X$. We again denote by~$\mathcal{M}$ the space of all nonexpansive self-mappings of $X$ and define on it the metric
\[
  d_z(f,g) = \sum_{n=1}^{\infty} 2^{-n} \frac{\rho(f(z_n),g(z_n))}{1+\rho(f(z_n),g(z_n))},
\]
which is indeed a metric because a nonexpansive mapping has a unique extension from a dense subset. Since we also want to employ Cauchy sequences, in addition to the topology of pointwise convergence we consider a natural uniformity generating this topology. With this aim we consider the uniformity consisting of all supersets of sets of the form
\[
  U_{F,\varepsilon} := \left\{(f,g)\in\mathcal{M}^2\colon \rho(f(z),g(z)) < \varepsilon\;\text{for all}\;z\in F\right\}
\]
where $F\subset X$ is finite and $\varepsilon>0$. A neighbourhood base of an $f\in\mathcal{M}$ for the topology generated by this uniformity consists of the sets
\[
  U_{F,\varepsilon}(f) = \{g\in\mathcal{M}\colon \rho(f(z),g(z))<\varepsilon \;\text{for all}\;z\in F\}.
\]
This shows that this topology is indeed the topology of pointwise convergence.

\begin{proposition}
  A sequence in $\mathcal{M}$ converges for $d_z$ if and only if it converges pointwise. In particular $(\mathcal{M},d_z)$ is a complete metric space.
\end{proposition}

\begin{proof}
  We start by showing that every sequence which is convergent for the above metric is pointwise convergent. To this end, let $x\in X$, $f\in \mathcal{M}$ and consider a sequence $(f_k)_{k=1}^{\infty}$ in~$\mathcal{M}$ which converges to $f$ with respect to the above metric. For every $\varepsilon>0$, we may pick a point $z_n$ with $\rho(x,z_n)<\frac{\varepsilon}{3}$. An argument similar to the one in the proof of Lemma~\ref{lem:localAndGlobalDistance} shows the existence of an $\alpha_n>0$ with the property that $d_z(f,g)<\alpha_n\varepsilon$ implies that $\rho(g(z_n),f(z_n))<\frac{\varepsilon}{3}$. Since $f_k\to f$ for the metric $d_z$, we may pick a $K\in\mathbb{N}$ such that $d_z(f_k,f)<\alpha_n\varepsilon$ for all $k\geq K$ and hence
  \[
    \rho(f_k(x),f(x)) \leq \rho(f_k(z_n),f(z_n)) + \frac{2\varepsilon}{3} < \varepsilon
  \]
  for all $k\geq K$. This shows that every $d_z$-convergent sequence is pointwise convergent.

  To prove the converse, assume that $f_k\to f$ pointwise. Given $\varepsilon>0$ we pick an $N\in\mathbb{N}$ such that
  \[
    \sum_{n=N}^{\infty} 2^{-n} < \frac{\varepsilon}{2}.
  \]
  Since $f_k\to f$ pointwise there is a $K\in\mathbb{N}$ with the property that $\rho(f_k(z_n),f(z_n)) < \frac{\varepsilon}{2}$ for $n=1,\ldots, N-1$ and all $k\geq K$. Combining these inequalities, we obtain
  \[
    d_z(f_k,f) \leq \sum_{n=1}^{N-1} 2^{-n} \rho(f_k(z_n),f(z_n)) + \sum_{n=N}^{\infty} 2^{-n} < \varepsilon
  \]
  for $k\geq K$.

  These arguments show that the topology of pointwise convergence and the one generated by $d_z$ have the same convergent sequences and a similar argument shows that the uniformity considered above has the same Cauchy sequences as the uniformity generated by $d_z$.
\end{proof}

\begin{remark}\label{rem:ptw}
  We make two observations concerning this proof for further reference.
  \begin{enumerate}
  \item A simple computation in the spirit of the above proof shows that the metric $d_z$ generates the topology of pointwise convergence on~$X$.
  \item In the above proof, we used that an argument similar to the one in the proof of Lemma~\ref{lem:localAndGlobalDistance} shows that for every $N\in\mathbb{N}$ there is an $\alpha_N>0$ with the property that for every  $\varepsilon\in(0,1)$, $d_z(f,g)\leq \alpha_N\varepsilon$ implies that $\rho(f(z_n),g(z_n))) < \varepsilon$ for all $n \leq N$.
  \end{enumerate}
\end{remark}

\begin{proposition}\label{prop:StrContrDensePW}
  The set of strict contractions is a dense subset of the space~$(\mathcal{M}, d_z)$.
\end{proposition}

\begin{proof}
  Given $r>0$, we pick an arbitrary point $\theta\in X$ and an $N\in\mathbb{N}$ with
  \[
    \sum_{n=N}^{\infty} 2^{-n} < \frac{r}{2}.
  \]
  We now choose a $\gamma\in (0,1)$ so that
  \[
    \rho((1-\gamma) f(z_n) \oplus \gamma f(\theta), f(z_n)) < \frac{r}{2}
  \]
  for $n=1,\ldots, N-1$. Using the hyperbolicity of $X$ together with the above bounds, we obtain that the mapping defined by
  \[
    f_\gamma (x)=(1-\gamma)f(x)\oplus \gamma f(\theta)
  \]
  is a $(1-\gamma)$-Lipschitz mapping the distance of which to~$f$ is at most~$r$.
\end{proof}

If we assume $X$ to be in addition \emph{unbounded}, similarly to the case of the metric of uniform convergence on bounded subsets, we obtain that for the metric $d_z$ the typical nonexpansive mapping is not a Rakotch contraction.

As in Section~\ref{sec:Rakotch}, for $\mu\in (0,1)$ and $t_0>0$, we set
\[
  \mathcal{N}_{\mu,t_0} := \{f\in\mathcal{M}\colon \omega_f(t_0) \leq \mu t_0\}
\]
and prove the following theorem.

\begin{theorem}
  For every $t_0>0$ and every $\mu\in(0,1)$, the set $\mathcal{N}_{\mu,t_0}$ is nowhere dense in~$(\mathcal{M}, d_z)$.
\end{theorem}

\begin{proof}
  Given $r\in(0,1)$ and $f\in\mathcal{M}$, we use Proposition~\ref{prop:StrContrDensePW} to obtain a strict contraction $\tilde{f}$ with $d_z(f,\tilde{f})<\frac{r}{2}$. We pick a large enough $N\in\mathbb{N}$ such that $2^{-N} < \frac{r}{8t_0}$ and a large enough radius $R_N\geq t_0$ such that $z_1,\ldots,z_{N-1}\in B(z_1,R_N)$. Using Lemma~\ref{lem:LargeModCont} for $\tilde{f}$, $z=z_1$, $R=R_N$ and $\lambda=\frac{1+\mu}{2}$, we obtain a nonexpansive mapping $g\colon X\to X$ which coincides with~$\tilde{f}$ on $B(z_1,R_N)$, satisfies the inequality $\rho(\tilde{f}(x),g(x)) \leq 2t_0$ otherwise and for which $\rho(g(x_0),g(y_0))> \frac{1+\mu}{2}t_0$.
  This allows us to conclude that
  \[
    d_z(f,g) \leq d_z(f,\tilde{f}) + \sum_{n=N}^{\infty} 2^{-n} \frac{\rho(\tilde{f}(z_n),g(z_n))}{1+\rho(\tilde{f}(z_n),g(z_n))} \leq \frac{r}{2} + 4t_0 2^{-N} < r.
  \]
  We set $\varepsilon :=\frac{t_0(1-\mu)}{8}$ and use the fact that the sequence $(z_n)_{n=1}^{\infty}$ is dense in $X$ to pick points $z_{n_1}$ and $z_{n_2}$ with $\rho(x_0,z_{n_1})<\frac{\varepsilon}{2}$ and $\rho(y_0,z_{n_2})<\frac{\varepsilon}{2}$. We pick an $\alpha>0$ such that $d_z(g,h)<\alpha\varepsilon=:\delta$ implies that $\rho(g(z_{n_1}),h(z_{n_1}))<\varepsilon$ and $\rho(g(z_{n_2}),h(z_{n_2}))<\varepsilon$; see Remark~\ref{rem:ptw}. In other words, for $h\in B(g,\delta)$ we have
  \begin{align*}
    \rho(h(x_0),h(y_0)) & \geq \rho(g(x_0),g(y_0)) - \rho(g(x_0),g(z_{n_1})) - \rho(g(z_{n_1}),h(z_{n_1})) - \rho(h(x_0),h(z_{n_1})) \\
                        & \qquad - \rho(g(y_0),g(z_{n_2})) - \rho(g(z_{n_2}),h(z_{n_2})) - \rho(h(y_0),h(z_{n_2}))\\
    & \geq \rho(g(x_0),g(y_0)) - 4 \frac{\varepsilon}{2} - 2 \varepsilon > \frac{1+\mu}{2}t_0 -  \frac{1-\mu}{2}t_0 = \mu t_0,
  \end{align*}
  that is, $h\not\in\mathcal{N}_{\mu,t_0}$. In other words, we have shown that for every $r\in(0,1)$ and every $f\in \mathcal{M}$ there are a $g\in B_{d_z}(f,r)$ and a $\delta>0$ such that $B_{d_z}(g,\delta)\subset \mathcal{M} \setminus \mathcal{N}_{\mu,t_0}$.
\end{proof}

The proofs of the following corollaries are completely analogous to the ones of their counterparts in Section~\ref{sec:Rakotch}.

\begin{corollary}
  The set $\{f\in\mathcal{M}\colon \omega_f(t) < t\;\text{for some}\;t>0\}$ is meagre in $(\mathcal{M}, d_z)$.
\end{corollary}

\begin{corollary}
  The set of Rakotch contractions is meagre in $(\mathcal{M}, d_z)$.
\end{corollary}

\begin{remark}
  In contrast to the situation of Theorem~\ref{thm:modcont}, a direct transfer of the proof of Theorems~\ref{thm:BallIntoBall} and~\ref{thm:RakotchOnBoundedSubsetsIsSigmaPorous} to this setting of pointwise convergence does not seem to work. The main problem is that here finitely many summands of the metric are not enough to have control over the values of a nonexpansive mapping on a (non-compact) ball. More precisely, for the proof of a result corresponding to Theorems~\ref{thm:BallIntoBall}, we can use the same construction to obtain for every $f\in\mathcal{M}$ and every $\varepsilon>0$ a mapping $g\in\mathcal{M}$ with distance to $f$ of at most $\varepsilon$ and a constant $M_g$ with
  \[
    g(\bar{B}(\theta,M_g)) \subset \bar{B}(\theta, M_g)
  \]
  Now the problem is that for $h\in\mathcal{M}$ with $d_z(g,h)<\alpha\varepsilon$, we can only conclude that~$g(x)$ and~$h(x)$ are sufficiently close for finitely many~$x$, which is not enough to show that $h$ satisfies a condition similar to the one above.
\end{remark}

On a positive note in the direction of Rakotch contractivity we have the following observation.

\begin{theorem}\label{thm:PtwShrink}
  There is a set $\mathcal{F}\subset\mathcal{M}$ such that $\mathcal{M}\setminus \mathcal{F}$ is meagre in $(\mathcal{M}, d_z)$ and for every $f\in\mathcal{F}$ the set
  \[
    \{(x,y)\in X\times X\colon x\neq y \;\text{and}\;\rho(f(x),f(y)) < \rho(x,y)\}
  \]
  is a residual subset of $X\times X$.
\end{theorem}

The proof of this theorem is based on the following lemma.

\begin{lemma}\label{lem:ptwShrink}
  For $x,y\in X$, $x\neq y$, we set
  \[
    \mathcal{F}_{x,y} := \left\{f\in\mathcal{M}\colon \exists r>0 \;\text{such that}\; \forall \xi\in B(x,r) \;\forall \eta\in B(y,r)\colon \rho(f(\xi),f(\eta))<\rho(\xi,\eta)\right\}.
  \]
  Then its complement $\mathcal{M}\setminus\mathcal{F}_{x,y}$ is a nowhere dense subset of $(\mathcal{M}, d_z)$.
\end{lemma}

\begin{proof}
  Let $\varepsilon>0$ and $f\in \mathcal{M}$ be given. By Proposition~\ref{prop:StrContrDensePW}, we may pick a strict contraction $g$ with $d_z(f,g)<\varepsilon$. We set $L:=\Lip g$ and
  \[
    r := \frac{(1-L)\rho(x,y)}{10}.
  \]
  Since $(z_n)_{n=1}^{\infty}$ is a dense sequence, we may pick $z_{m_1}$ and $z_{m_2}$ with $\rho(x,z_{m_1})<r$ and $\rho(y,z_{m_2})<r$. We pick an $\alpha>0$ such that  $\rho(h(z_{m_2}),g(z_{m_2})) < r$ and $\rho(h(z_{m_1}),g(z_{m_1}))<r$ whenever $d_z(g,h)<\alpha r$; see Remark~\ref{rem:ptw}. We set
  \[
    \delta := \alpha r,
  \]
  and let $h\in\mathcal{M}$ with $d_{z}(g,h)<\delta$. Given $\xi\in B(x,r)$ and $\eta\in B(y,r)$ we use the triangle inequality, the fact that $h$ is nonexpansive and the above bound on the distances of $g$ and $h$ at $z_{m_1}$ and $z_{m_2}$ to obtain that
  \begin{align*}
    \rho(h(\xi),h(\eta)) & \leq \rho(h(\xi),h(x)) + \rho(h(x),h(z_{m_1})) + \rho(h(z_{m_1}),g(z_{m_1})) + h(g(z_{m_1}),g(z_{m_2}))\\ & \qquad + \rho(g(z_{m_2}),h(z_{m_2}))  + \rho(h(z_{m_2}),h(y)) + \rho(h(y),h(\eta))\\
                         & \leq 3r + \rho(g(z_{m_1}),g(z_{m_2})) + 3r \leq 6r + L \rho(z_{m_1},z_{m_2}) \leq 6r + L(\rho(x,y) + 2r) \\
                         & < 8r + L \rho(x,y) = (1-L) \rho(x,y) - 2r + L \rho(x,y) = \rho(x,y) -2 r \leq \rho(\xi,\eta)
  \end{align*}
  which shows that $h\in \mathcal{F}_{x,y}$. In other words, we have shown that for every $\varepsilon>0$ and every $f\in \mathcal{M}$ there is a $g\in B_{d_z}(f,\varepsilon)$ and a $\delta>0$ such that $B_{d_z}(g,\delta)\subset \mathcal{F}_{x,y}$.
\end{proof}

\begin{proof}[Proof of Theorem~\ref{thm:PtwShrink}]
  We set
  \[
    \mathcal{F} := \bigcap_{m,n\in\mathbb{N}} \mathcal{F}_{z_{m},z_{n}}
  \]
  where $\mathcal{F}_{z_{m},z_{n}}$ is the set defined in Lemma~\ref{lem:ptwShrink}. Given $f\in \mathcal{F}$ we use Lemma~\ref{lem:ptwShrink} to obtain for every $(m,n)\in\mathbb{N}^2$ a radius $r_{m,n}>0$ such that $\rho(f(\xi),f(\eta)) < \rho(\xi,\eta)$ for all pairs $(\xi,\eta) \in B(z_m,r_{m,n})\times B(z_n,r_{m,n})$. Hence the set
  \[
    \{(x,y)\in X\times X\colon x\neq y\;\text{and}\;\rho(f(x),f(y)) < \rho(x,y)\}
  \]
  contains the dense open set
  \[
    \bigcup_{\substack{n,m\in\mathbb{N}\\m\neq n}} B(z_{n},r_{n,m}) \times B(z_{m}, r_{n,m})
  \]
  and it is a residual subset of~$X$.
\end{proof}

\medskip\medskip{}

\noindent\textbf{Acknowledgements.} All the authors are grateful to two anonymous referees for their useful comments and helpful suggestions. In particular, they greatly appreciate the suggestions regarding Theorems~\ref{thm:BallIntoBall} and~\ref{thm:modcont} and the very close reading of the article. The authors thank Michael Dymond for several discussions on this topic and Eva Kopeck\'{a} for pointing out Lemma~2 of~\cite{IvesPreiss}. The research of the first and the third author is supported by the Austrian Science Fund (FWF): P~32523-N. The second author was partially supported by the Israel Science Foundation (Grant 820/17), the Fund for the Promotion of Research at the Technion and by the Technion General Research Fund.

\end{document}